\RequirePackage{fix-cm}

\documentclass[smallcondensed]{svjour3}

\smartqed
\usepackage{float}
\usepackage{color}
\usepackage{cite}
  \usepackage{mathrsfs}

\usepackage{comment}
\usepackage{amsmath}
\usepackage{amssymb}
\usepackage{multirow}
\usepackage{graphicx}
\usepackage[ruled,vlined]{algorithm2e}
\usepackage{hyperref}
\usepackage{algorithmic}

\usepackage{graphicx}
\usepackage[ruled,vlined]{algorithm2e}

\usepackage[misc]{ifsym}

\newcommand{\bq }{\begin{equation}}
\newcommand{\eq }{\end{equation}}

\numberwithin{equation}{section}

\usepackage{tikz,xcolor,hyperref}

\definecolor{lime}{HTML}{A6CE39}
\DeclareRobustCommand{\orcidicon}{
	\begin{tikzpicture}
	\draw[lime, fill=lime] (0,0)
	circle [radius=0.16]
	node[white] {{\fontfamily{qag}\selectfont \tiny ID}};
	\draw[white, fill=white] (-0.0625,0.095)
	circle [radius=0.007];
	\end{tikzpicture}
	\hspace{-2mm}
}
\foreach \x in {A, ..., Z}{\expandafter\xdef\csname orcid\x\endcsname{\noexpand\href{https://orcid.org/\csname orcidauthor\x\endcsname}
			{\noexpand\orcidicon}}
}

\begin{document}

\title{A two-grid temporal second-order scheme for the two-dimensional nonlinear Volterra integro-differential equation with weakly singular kernel}

\author{Hao Chen \and Mahmoud A. Zaky \\   Ahmed S. Hendy \and Wenlin Qiu}

\institute{H. Chen \at
MOE-LCSM, School of Mathematics and Statistics, Hunan Normal University, Changsha, Hunan 410081, P. R. China\\
\email{haochen9604@163.com}\\
M. A. Zaky \at
Department of Applied Mathematics, National Research Centre, Dokki, Cairo 12622, Egypt\\
\email{ma.zaky@yahoo.com}\\
A. S. Hendy \at
Department of Computational Mathematics and Computer Science, Institute of Natural Sciences and Mathematics, Ural Federal University, 19 Mira St., Yekaterinburg 620002, Russia;
\at
Department of Mathematics, Faculty of Science,
Benha University, Benha 13511, Egypt\\
\email{ahmed.hendy@fsc.bu.edu.eg}\\
W. Qiu (\Letter) \at
MOE-LCSM, School of Mathematics and Statistics, Hunan Normal University, Changsha, Hunan 410081, P. R. China\\
\email{qwllkx12379@163.com}
}
\date{Received: date / Accepted: date}

\maketitle

\begin{abstract}
In this paper, a two-grid temporal second-order  scheme for the two-dimensional nonlinear Volterra integro-differential equation with weakly singular kernel is proposed to   reduce the computation time and improve the accuracy of the scheme developed by Xu et al. (Applied Numerical Mathematics 152 (2020) 169-184).
The proposed  scheme consists of three steps:  First, a small nonlinear system is solved on the coarse grid using fix-point iteration. Second, the Lagrange's linear interpolation formula is used to arrive at some auxiliary values for analysis of the fine grid. Finally, a linearized Crank-Nicolson finite difference system is solved on the fine grid. Moreover, the algorithm uses a central difference approximation for the spatial derivatives. In the time direction, the time derivative and integral term are approximated by Crank-Nicolson technique and product integral rule, respectively.  With the help of the discrete energy method, the stability and space-time second-order convergence of the proposed approach are obtained in $L^2$-norm. Finally, the numerical results agree with the theoretical analysis and verify the effectiveness of the algorithm.
\end{abstract}

 \section{Introduction}
 \vskip 0.2mm
 In this paper, we consider the following two-dimensional nonlinear Volterra integro-differential equation with weakly singular kernel
 \begin{equation}\label{eq1.1}
   u_{t}-\mu\Delta u-I^{(\alpha)}\Delta u=f(x,y,t)+g(u),\quad (x,y,t)\in\Omega \times (0,T],
 \end{equation}
 with the initial-boundary conditions
 \begin{equation}\label{eq1.3}
   \begin{split}
       &u(x,y,0)=\psi(x,y),\qquad (x,y)\in\bar{\Omega},\\
       &   u(x,y,t)=0,\qquad (x,y,t)\in\partial\Omega\times (0,T],
   \end{split}
 \end{equation}
 where $\Omega=(0,L_{x})\times(0,L_{y})$ with the  boundary $\partial\Omega$, $\Delta=\partial^{2}/\partial x^{2}+\partial^{2}/\partial y^2$ is the two-dimensional Laplacian operator and $u_t=\partial u/\partial t$. In addition, $\alpha\in(0,1)$, $\mu\in[0,\infty)$ and $T\in(0,\infty)$ are given constants. $f(x,y,t)$ and $\psi(x,y)$ are given functions. The nonlinear term $g(u)\in C^{2}(\mathbf{R})\cap L^{1}(0,T]$ satisfies the Lipschitz condition $|g(u_1)-g(u_2)|\leq \bar{C}|u_1-u_2|$. Furthermore, The integral term $I^{(\alpha)}\Delta u(x,y,t)$ is defined \cite{podlubny1999fractional,qiao2021second} as follows
 \begin{equation}\label{eq1.4}
   I^{(\alpha)}\Delta u(x,y,t)=\int^t_0\rho_\alpha(t-s)\Delta u(x,y,s)ds,\quad \rho_\alpha(t)=\frac{t^{\alpha-1}}{\Gamma(\alpha)},\quad t> 0.
 \end{equation}
 \vskip 0.2mm
    In addition, throughout the article, we assume that problem \eqref{eq1.1}-\eqref{eq1.3} has a unique solution such that the following regularity assumptions \cite{mustapha2010second}:
    \begin{itemize}
        \item [\textbf{(A1)}] $u_t$, $u_{tyy}$, $u_{txx}$, $u_{xxxx}$ and $u_{yyyy}$ are continuous in $\bar{\Omega}\times [0,T]$;
      \item [\textbf{(A2)}] $u_{tt}$, $u_{ttt}$, $u_{ttxx}$ and $u_{ttyy}$ are continuous in $\bar{\Omega}\times (0,T]$, and there exists a positive constant $\bar{C}$ satisfying for $(x,y,t)\in\bar{\Omega}\times (0,T]$ that
     \begin{equation*}
     \begin{split}
  &   |u_{tt}(x,y,t)|\leq \bar{C}t^{\alpha-1}, \qquad |u_{ttt}(x,y,t)|\leq \bar{C}t^{\alpha-2}, \\
     &|u_{tt\kappa\kappa}(x,y,t)|\leq \bar{C}t^{\alpha-1}(\kappa=x,y).
     \end{split}
     \end{equation*}
    \end{itemize}
     Such integro-differential equations with Riemann-Liouville integral operators appear frequently in various mathematical and physical models. Problem \eqref{eq1.1}-\eqref{eq1.3} is a commonly used model for studying physical phenomena related to elastic forces. This model is mainly used in the problems of heat conduction, viscoelasticity and population dynamics of materials with memory \cite{friedman1967volterra,gurtin1968general,miller1978integrodifferential}. In viscoelastic problems, the parameter $\mu$ in this model represents the Newtonian contribution to viscosity, and the integral term represents the viscosity part of the equation.

     In recent years, high-precision computational methods for 2D partial integro-differential equations with weakly singular kernel, such as equation \eqref{eq1.1}, have been developed. The linear case of \eqref{eq1.1}-\eqref{eq1.3} has been deeply studied in the literature, e.g., see \cite{chen2017second,khebchareon2015alternating,kim1998spectral,larsson1998numerical,li2013alternating,wang2022crank}. Furthermore,   some numerical studies on the nonlinear case were introduced.
     Mustapha et al. \cite{mustapha2010second} applied the Crank-Nicolson scheme under graded meshes to solve semilinear integro-differential equation with weakly singular kernel. Dehghan et al. \cite{dehghan2017spectral} proposed a spectral element technique for solving nonlinear fractional evolution equation. In addition, some numerical methods for nonlinear partial differential equations have been proposed, and we can refer to the work in \cite{jiang2020adi,hendy2022energy,liao2019unconditional}.

     However, when solving 2D nonlinear problems, the resulting large systems of nonlinear equations require a large computational cost as the grid is continuously subdivided. In order to save the computational cost of nonlinear problems, a spatial two-grid finite element technique was proposed by Xu \cite{xu1996two,xu1994novel}.  Inspired by Xu's ideas, the two-grid method began to be intensively studied and applied to the solution of nonlinear parabolic equations. Dawson and Wheeler et al. \cite{dawson1998two} proposed a spatial two-grid finite difference method in solving nonlinear parabolic equations and analyzed the convergence of the method on coarse and fine grid. For solving the nonlinear time-fractional parabolic equation, Li et al. \cite{li2017two} obtained the numerical solution of this equation using the spatial two-grid block-centered finite difference scheme. For more work regarding the spatial two-grid methods, see \cite{bajpai2014two,chen2011two,chen2010two}. In addition, some scholars, inspired by the spatial two-grid method, started to consider using the two-grid method to solve the nonlinear equations in the time direction. Liu et al. \cite{liu2018time} proposed a new time two-grid finite element algorithm in order to solve the time fractional water wave model, and illustrated through numerical experiments that it has higher computational efficiency than the standard finite element method. In \cite{xu2020time}, a time two-grid backward Euler finite difference method is constructed to solve problem \eqref{eq1.1}-\eqref{eq1.3}. However, the time convergence order of the above methods cannot reach the exact second order.

     In this paper, we design an efficient  temporal two-grid Crank-Nicolson (TTGCN) finite difference method for solving problem \eqref{eq1.1}-\eqref{eq1.3}. In this approach, the time and space derivatives are approximated using the Crank-Nicolson technique and the central difference formula, respectively, and the Riemann-Liouville integral term is approximated by the product integration rule designed in \cite{mclean2007second}. Then, this algorithm is divided into three steps: First, a small nonlinear system is solved on a coarse grid. Second, based on the solution of the first step, the values of each node are obtained by linearization technique as the auxiliary approximate solution. Finally, we approximate the nonlinear term $g(U^{n})$ by a Taylor expansion and solve the linear system on a fine grid. Furthermore, under the regularity assumptions \textbf{(A1)} and \textbf{(A2)}, we prove that this algorithm is stability and the convergence of order $O(\tau_C^4+\tau_F^2+h_1^2+h_2^2)$, where $\tau_C$ and $\tau_F$ are the time steps of the coarse and fine grid, respectively. Also, the linearization technique is used on the fine grid, so the TTGCN finite difference algorithm has the advantage of both ensuring accuracy and improving computational efficiency. In addition, the numerical results in this paper show that the TTGCN finite difference algorithm is more efficient than the standard Crank-Nicolson (SCN) finite difference method without loss of accuracy. Meanwhile, our algorithm can achieve second-order convergence in time compared to the method in \cite{xu2020time}.

     The remainder of this paper is structured as follows. In Section \ref{sec2}, we give some notations and useful lemmas. Then, the TTGCN finite difference scheme is established in Section \ref{secc3}. In Section \ref{secc4}, the stability and convergence of the TTGCN finite difference method are analyzed by the energy method. Moreover, some numerical results are given in Section \ref{secc5}.

The generic positive constant $\bar{C}$ is independent of  the temporal step size and the spatial step size, moreover, it is not necessarily same in different situations.

   \section{Preliminaries}\label{sec2}

   \vskip 0.2mm
     In this section, we shall provide some useful notations and lemmas which will be used for the forthcoming work. First, for a positive integer $\mathcal{N}$, we define the time-step size on the fine grid as $\tau=\tau_{F}=T/\mathcal{N}$ with $t_{n}=n\tau_{F}(0\leq n\leq \mathcal{N})$. Similarly, for the coarse grid, the time-step size is $\tau_{C}=T/N$, $t_{s}=s\tau_{C}(0\leq s\leq N)$ for positive integer $N$, where $N=\mathcal{N}/k$, $k\geq 2$ and $k\in \mathbb{Z^{+}}$. For any grid function $\varphi^{n}(1\leq n\leq \mathcal{N})$ on $(0,T]$, we define
     \begin{equation*}
       \delta_{t}\varphi^{n}=\frac{\varphi^{n}-\varphi^{n-1}}{\tau},\qquad \varphi^{n-\frac{1}{2}}=\frac{\varphi^{n}+\varphi^{n-1}}{2}.
     \end{equation*}
     Then, we define the grid functions as following
      \begin{equation*}
     u^{n}=u(x,y,t_{n}), \quad f^{n}=f(x,y,t_{n}), \quad 0\leq n \leq \mathcal{N}.
     \end{equation*}
     We integrate the equation (\ref{eq1.1}) from $t=t_{n-1}$ to $t_{n}$ and then multiply by $\frac{1}{\tau}$, we obtain
     \begin{equation}\label{eq2.1}
      \delta_{t}u^{n}-\frac{\mu}{\tau}\int\limits^{t_{n}}_{t_{n-1}}\Delta u(\cdot,t)dt -\frac{1}{\tau}\int\limits^{t_{n}}_{t_{n-1}}I^{(\alpha)}\Delta u(\cdot,t)dt
      =\frac{1}{\tau}\int\limits^{t_{n}}_{t_{n-1}}f(\cdot,t)dt+ \frac{1}{\tau}\int\limits^{t_{n}}_{t_{n-1}}g(u(\cdot,t))dt.
     \end{equation}

      \vskip 0.2mm

     To approximate the integral term of equation \eqref{eq2.1}, from \cite{mclean2007second,wang2022crank}, we obtain the quadrature approximation with the uniform time step
     \begin{equation}\label{eq2.2}
       \begin{array}{ll}
        \frac{1}{\tau}\int\limits^{t_{n}}_{t_{n-1}}\Delta u(\cdot,t)dt=
        \begin{cases}
          \Delta u^{1}+(R1)^{1}, \\
          \Delta u^{n-\frac{1}{2}}+(R1)^{n},\quad  2\leq n\leq \mathcal{N},
        \end{cases}
       \end{array}
     \end{equation}
     and
     \begin{equation}\label{eq2.3}
     \begin{array}{lll}
     \frac{1}{\tau}\int\limits^{t_{n}}_{t_{n-1}} I^{(\alpha)}\Delta u(\cdot,t)dt =\\
     \quad
     \begin{cases}
       \frac{1}{\tau}\int\limits^{t_{1}}_{t_{0}}\int\limits_{t_{0}}^{t_{1}}\rho_{\alpha}(t-s)\Delta u^{1}dsdt+(R2)^{1},\\
       \\
       \frac{1}{\tau}\int\limits^{t_{n}}_{t_{n-1}}\int\limits_{t_{0}}^{t_{1}}\rho_{\alpha}(t-s)\Delta u^{1}dsdt\\
       +\frac{1}{\tau}\int\limits^{t_{n}}_{t_{n-1}}\sum\limits^{n}_{m=2}
       \int\limits_{t_{m-1}}^{\min\{t,t_{m}\}}\rho_{\alpha}(t-s)\Delta u^{m-\frac{1}{2}}dsdt+(R2)^{n}, \quad  2\leq n\leq \mathcal{N},
     \end{cases}
     \end{array}
     \end{equation}
     where $(R1)^{n}$ and $(R2)^{n}$ are the local truncation errors.

   \vskip 0.2mm
   For any grid function $\varphi^{n}(1\leq n\leq \mathcal{N})$, we define the following two operators
   \begin{equation}\label{eq2.4}
     \begin{array}{lll}
     \mathfrak{L}_{1,\tau}^{n}(\varphi^{n})=
     \begin{cases}
       \varphi^{1}, &  \\
       \varphi^{n-\frac{1}{2}}, & n\geq 2,\quad
     \end{cases}
     \\
     \mathfrak{L}_{2,\tau}^{n}(\varphi^{n})=
     \begin{cases}
       \mathfrak{w}_{1,1}\varphi^{1}, &  \\
       \mathfrak{w}_{n,1}\varphi^{1}+\sum\limits^{n}_{m=2}\mathfrak{w}_{n,m}\varphi^{m-\frac{1}{2}}, & n\geq 2,
     \end{cases}

     \end{array}
     \end{equation}
   where
   \begin{equation}\label{eq2.5}
     \mathfrak{w}_{n,m}=\frac{1}{\tau}\int\limits_{t_{n-1}}^{t_{n}} \int\limits_{t_{m-1}}^{\min\{t,t_{m}\}}\rho_{\alpha}(t-s)dsdt.
   \end{equation}

   Therefore, for $n\geq 2$ and $1\leq m\leq n-1$, we can get that
   \begin{equation}\label{eq2.6}
   \begin{split}
     & \mathfrak{w}_{n,m}  = \frac{1}{\tau}\int\limits^{t_{n}}_{t_{n-1}}\int\limits_{t_{m-1}}^{t_{m}}
     \rho_{\alpha}(t-s)dsdt \\
        &=\frac{\left[(t_{n}-t_{m-1})^{\alpha+1}-(t_{n}-t_{m})^{\alpha+1}\right] -\left[(t_{n-1}-t_{m-1})^{\alpha+1}-(t_{n-1}-t_{m})^{\alpha+1}\right]}{\tau\Gamma(2+\alpha)}
   \end{split}
     \end{equation}
   and
   \begin{equation}\label{eq2.7}
     \mathfrak{w}_{n,n}=\frac{1}{\tau}\int\limits^{t_{n}}_{t_{n-1}}\int\limits_{t_{n-1}}^{t}
     \rho_{\alpha}(t-s)dsdt=\frac{\tau^{\alpha}}{\Gamma(2+\alpha)},\qquad 1\leq n\leq \mathcal{N}.
     \end{equation}
   Then the equations \eqref{eq2.2} and \eqref{eq2.3} can be rewritten as follows
   \begin{equation}\label{eq2.8}
     \frac{1}{\tau}\int\limits_{t_{n-1}}^{t_{n}}\Delta u(\cdot,t)dt = \mathfrak{L}_{1,\tau}^{n}(\Delta u^{n})+(R1)^n, \qquad 1\leq n \leq \mathcal{N},
   \end{equation}
   \begin{equation}\label{eq2.9}
     \frac{1}{\tau}\int\limits_{t_{n-1}}^{t_{n}}I^{(\alpha)}\Delta u(\cdot,t)dt=
     \mathfrak{L}_{2,\tau}^{n}(\Delta u^{n})+(R2)^{n}, \qquad 1\leq n \leq \mathcal{N}.
     \end{equation}
  \vskip 0.2mm
  For the spatial approximation, defining the space-step size $h_1=L_x/M_x$, $h_2= L_y/M_y$, $h = \max\{h_1,h_2\}$ for two positive integers $M_x$ and $M_y$, we arrive at $x_i = ih_1$ and $y_j = jh_2$. Denote $ \bar{\Omega}_h = \{(x_i,y_j)| 0\leq i \leq M_x, 0\leq j \leq M_y \}$, $\Omega_h = \bar{\Omega}_h \cap \Omega$ and $\partial\Omega_h= \Omega_h\cap \partial \Omega$. Let the grid function $Z_{h}=\{z_{ij}|0\leq i\leq M_{x},0\leq j\leq M_{y}\}$ on $\Omega_h$, then we denote the following notations
  \begin{equation*}
  \begin{split}
     \delta_{x}z_{i+\frac{1}{2},j}=\frac{z_{i+1,j}-z_{ij}}{h_{1}}, & \qquad \delta_{x}^{2}z_{ij}= \frac{\delta_{x}z_{i+\frac{1}{2},j}-\delta_{x}z_{i-\frac{1}{2},j}}{h_1}, \\
     \delta_{y}z_{i,j+\frac{1}{2}}=\frac{z_{i,j+1}-z_{ij}}{h_{2}}, & \qquad \delta_{y}^{2}z_{ij}= \frac{\delta_{y}z_{i,j+\frac{1}{2}}-\delta_{y}z_{i,j-\frac{1}{2}}}{h_2}.
     \end{split}
  \end{equation*}
  Also, the discrete Laplace operator is defined by $\Delta_h=\delta_x^2+\delta_y^2$.
  \vskip 0.2mm
  Then, for any grid function $z, v\in \Omega_h$, some norm and inner product are defined as follows
    \begin{equation*}
  \begin{split}
   (z,v) = h_1h_2\sum_{i=1}^{M_x-1}\sum_{j=1}^{M_y-1}z_{ij}v_{ij}, \quad \| z \| =  \sqrt{( z,z)}, \quad \| z \|_{\infty} = \max_{\substack{1\leq i\leq M_x-1, \\ 1\leq j\leq M_y-1}}|z_{ij}|,\\
    \| \delta_xz \| = \sqrt{h_1h_2\sum_{i=0}^{M_x-1}\sum_{j=1}^{M_y-1}(\delta_x z_{i+\frac{1}{2},j})^2},\quad \| \delta_y z \| = \sqrt{h_1h_2 \sum_{i=1}^{M_x-1}\sum_{j=0}^{M_y-1}(\delta_x z_{i,j+\frac{1}{2}})^2}.
     \end{split}
  \end{equation*}

   \vskip 0.2mm
   Next, some auxiliary lemmas will be given.
   \begin{lemma}\label{lemma2.1}\cite{dedic2001euler} Suppose $g(u(\cdot,t))\in C^{2}(\mathbf{R})\cap L^1(0,T)$, then it holds that
   \begin{equation*}
     \left|\int^{t_n}_{t_{n-1}}g(u(\cdot,t))dt-\frac{t_{n}-t_{n-1}}{2}\Big[g(u(\cdot,t_n)) +g(u(\cdot,t_{n-1}))\Big]\right|\leq \frac{(t_{n}-t_{n-1})^3}{12}\|g^{\prime\prime}\|_{\infty},
   \end{equation*}
   where $\|g^{\prime\prime}\|_{\infty} =\sup\limits_{\xi\in(t_{n-1},t_{n})}|g^{\prime\prime}(u(\cdot,\xi))|<\infty$.
   \end{lemma}
   \vskip 0.2mm
   According to Taylor series expansion with integral remainder term, we can obtain the following lemma.
   \vskip 0.2mm
   \begin{lemma}\label{lemma2.2} \cite{chen2015alternating} Assume $v(x,y)\in C_{x,y}^{4,4}([0,L_x]\times[0,L_y])$, then it satisfies that
    \begin{equation*}
      \frac{\partial^2v}{\partial x^2}(x_i,y_j) =\delta_x^2v(x_i,y_j)
      - \frac{h_1^2}{6}\int_0^1\biggl[\frac{\partial^4v}{\partial x^4}(x_i+wh_1,y_j)+\frac{\partial^4v}{\partial x^4}(x_i-wh_1,y_j)\biggl](1-w)^3dw,
    \end{equation*}
    \begin{equation*}
    \frac{\partial^2v}{\partial y^2}(x_i,y_j) = \delta_y^2v(x_i,y_j)
    - \frac{h_2^2}{6}\int_0^1\biggl[\frac{\partial^4v}{\partial y^4}(x_i,y_j+wh_2)+\frac{\partial^4v}{\partial y^4}(x_i,y_j-wh_2)\biggl](1-w)^3dw.
    \end{equation*}
    \end{lemma}

   \vskip 0.2mm
   For further analysis, we present the following important lemmas.
   \begin{lemma}\label{lemma2.3} Assume that the solution $u$ of the problem (\ref{eq1.1})-(\ref{eq1.3}) satisfies the regularity assumptions \textbf{(A1)} and \textbf{(A2)}, then we obtain that
   \begin{equation*}
   \begin{array}{ll}
       \tau\sum\limits^{n}_{m=1}\|(R1)^{m}\|\leq \bar{C}\tau^{2}, \qquad 1\leq n \leq \mathcal{N}.
   \end{array}
   \end{equation*}
   \end{lemma}
   \begin{proof} Through simple calculation, we yield
   \begin{equation}\label{eq2.10}
     (R1)^{1}=\frac{1}{\tau}\int\limits^{t_{1}}_{t_0}\left[\Delta u(\cdot,t)-\Delta u^{1}\right]dt,
   \end{equation}
   \begin{equation}\label{eq2.11}
     (R1)^{n}=\frac{1}{\tau}\int\limits^{t_{n}}_{t_{n-1}}\left[\Delta u(\cdot,t)-\left(\frac{t_{n}-t}{\tau}\Delta u^{n-1}+\frac{t-t_{n-1}}{\tau}\Delta u^{n}\right)\right]dt.
   \end{equation}
   Using Taylor expansion with integral remainder term, we have
   \begin{equation}\label{eq2.12}
     \Delta u(\cdot,t)-\Delta u^{1}=\Delta u(\cdot,t)-\Delta u(\cdot,t_{1})=-\int\limits^{t_{1}}_{t}\Delta u_s(\cdot,s)ds,\quad t_{0}\leq t\leq t_{1},
   \end{equation}
   therefore
   \begin{equation}\label{eq2.13}
     (R1)^{1}=-\frac{1}{\tau}\int\limits^{t_1}_{t_0}\int\limits^{t_1}_{t}\Delta u_s(\cdot,s)dsdt=-\frac{1}{\tau}\int\limits^{t_1}_{t_0}\int\limits^{s}_{t_0}\Delta u_s(\cdot,s)dtds=-\frac{1}{\tau}\int\limits^{t_1}_{t_0}s\Delta u_s(\cdot,s)ds.
   \end{equation}
   The continuity of $u_{t\kappa\kappa}(x,y,t)(\kappa=x,y)$ in $\bar{\Omega}\times[0,T]$ implies
   \begin{equation}\label{eq2.14}
     \tau\|(R1)^{1}\|\leq \bar{C}\tau^{2}.
   \end{equation}
  Similarly, from Taylor expansion with integral remainder term, we obtain
   \begin{equation}\label{eq2.15}
     \left|\Delta u(\cdot,t)-\frac{t_{n}-t}{\tau}\Delta u^{n-1}-\frac{t-t_{n-1}}{\tau}\Delta u^{n}\right|\leq 2\tau\int\limits^{t_n}_{t_{n-1}}|\Delta u_{ss}(\cdot,s)|ds,\quad n\geq 2,
   \end{equation}
   then
   \begin{equation}\label{eq2.16}
     |(R1)^{n}|\leq 2\int\limits^{t_n}_{t_{n-1}}\int\limits^{t_n}_{t_{n-1}}\left| \Delta u_{ss}(\cdot,s)\right|dsdt\leq \bar{C}\tau\int\limits^{t_n}_{t_{n-1}}s^{\alpha-1}ds= \bar{C}\tau(t_{n}^{\alpha}-t_{n-1}^{\alpha}).
   \end{equation}
   This proves
   \begin{equation}\label{eq2.17}
     \tau\sum\limits^{n}_{m=2}\|(R1)^{m}\|\leq \bar{C}\tau^{2}(t_{n}^{\alpha}-t_{1}^{\alpha})\leq \bar{C}\tau^2.
   \end{equation}
   The proof is completed.
   \end{proof}

   \vskip 0.2mm
   \begin{lemma}\label{lemma2.4} Suppose that the solution $u$ of the problem (\ref{eq1.1})-(\ref{eq1.3}) satisfies the regularity assumptions \textbf{(A1)} and \textbf{(A2)}. Then we can obtain the following
   \begin{equation}\label{eq2.18}
   \begin{array}{ll}
       \tau\sum\limits^{n}_{m=1}\|(R2)^{m}\|\leq \bar{C}\tau^{2}, \qquad 1\leq n \leq \mathcal{N}.
   \end{array}
   \end{equation}
   \end{lemma}
   \begin{proof} See the case $(\gamma=1)$ in \cite{mclean2007second}, or Lemma $2.2$ in \cite{wang2022crank}.
   \end{proof}

    \vskip 0.2mm
   \begin{lemma}\label{lemma2.5}\cite{chen2015alternating} For any grid function $v,w\in Z_{h}$, then it holds as follows
   \begin{equation*}
   \begin{split}
      -h_1h_2\sum_{i=1}^{M_x-1}\sum_{j=1}^{M_y-1}(\delta_{x}^{2}v_{ij})w_{ij}= h_1h_2\sum_{i=0}^{M_x-1}\sum_{j=1}^{M_y-1}(\delta_{x}v_{i+\frac{1}{2},j}) (\delta_{x}w_{i+\frac{1}{2},j}),\\
     -h_1h_2\sum_{i=1}^{M_x-1}\sum_{j=1}^{M_y-1}(\delta_{y}^{2}v_{ij})w_{ij}= h_1h_2\sum_{i=1}^{M_x-1}\sum_{j=0}^{M_y-1}(\delta_{y}v_{i,j+\frac{1}{2}}) (\delta_{y}w_{i,j+\frac{1}{2}}).
   \end{split}
   \end{equation*}
   \end{lemma}

   \vskip 0.2mm
   \begin{lemma}\label{lemma2.6}\cite{mclean2007second,qiao2021second} For any grid function $v^{n}(1\leq n \leq \mathcal{N})$, it holds that
   \begin{equation}\label{eq2.19}
   \begin{array}{ll}
      \big(\nabla_h v^{1},\mathfrak{L}_{2,\tau}^{1}(\nabla_h v^{n})\big)+\sum\limits^{\mathcal{N}}_{n=2} \big(\nabla_h v^{n-\frac{1}{2}},\mathfrak{L}_{2,\tau}^{n}(\nabla_h v^{n})\big)\geq 0,
   \end{array}
   \end{equation}
   where $\mathfrak{L}_{2,\tau}^{n}$ is presented via (\ref{eq2.4}) and the operator $\nabla_{h}=\delta_x+\delta_y$.
   \end{lemma}

   \vskip 0.2mm
   \begin{lemma}\label{lemma2.7} \cite{qiao2021second} For $\mathcal{N}\geq 1$ and $v^{n}\in Z_{h}$, we have
   \begin{equation}\label{eq2.20}
   \begin{array}{ll}
        \tau\left(v^{1},\delta_{t}v^1\right)+ \tau\sum\limits^{\mathcal{N}}_{n=2}\left(v^{n-\frac{1}{2}},\delta_{t}v^{n}\right)\geq \frac{1}{2}\left(\left\|v^{\mathcal{N}}\right\|^{2}- \left\|v^{0}\right\|^{2}\right).
   \end{array}
   \end{equation}
   \end{lemma}
  \vskip 0.2mm
\begin{lemma}\label{lemma2.8}{\cite{sloan1986time}} (Discrete Gr\"{o}nwall's inequality) If $\{\mathcal{Q}_m\}$ is a non-negative real sequence and satisfies
     \begin{equation*}
     	 \begin{array}{ll}
          \mathcal{Q}_m \leq \tilde{\gamma}_m + \sum\limits_{n=0}^{m-1}\tilde{\beta}_n \, \mathcal{Q}_n,  \qquad m \geq 1,
     \end{array}
     \end{equation*}
where $\{\tilde{\gamma}_m\}$ is a non-negative and non-descending sequence, $\tilde{\beta}_n \geq 0$, then, we obtain
     \begin{equation*}
     	 \begin{array}{ll}
          \mathcal{Q}_m \leq \tilde{\gamma}_m \, \exp(\sum\limits_{n=0}^{m-1}\tilde{\beta}_n ),  \qquad m \geq 1.
     \end{array}
     \end{equation*}
\end{lemma}
   \section{Establishment of the two-grid difference scheme}\label{secc3}
  In the following, we first establish the SCN finite difference method for nonlinear problem \eqref{eq1.1}-\eqref{eq1.3}.
  \vskip 0.2mm
    Applying the quadrature approximations \eqref{eq2.2}-\eqref{eq2.3} and Lemmas \ref{lemma2.1}-\ref{lemma2.2}, then \eqref{eq2.1} become
     \begin{equation}\label{eq3.1}
     \begin{split}
        \delta_{t}u^{1}_{ij}&-\mu\Delta_{h} u^{1}_{ij}- \mathfrak{w}_{1,1}\Delta_{h} u^{1}_{ij} =b^{1}_{ij}+\frac{g(u^1_{ij})+g(u^{0}_{ij})}{2}\\&+(R1)^{1}_{ij} +(R2)^{1}_{ij}+(R3)^{1}_{ij}+(R4)^{1}_{ij},\qquad (x_i,y_j)\in\Omega_{h},
      \end{split}
   \end{equation}
   \begin{equation}\label{eq3.2}
   \begin{split}
      \delta_{t}u^{n}_{ij}&-\mu\Delta_{h} u^{n-\frac{1}{2}}_{ij}- \mathfrak{w}_{n,1}\Delta_{h} u^{1}_{ij}-\sum_{m=2}^{n}\mathfrak{w}_{n,m}\Delta_{h} u^{m-\frac{1}{2}}_{ij}=b^{n}_{ij}+\frac{g(u^n_{ij})+g(u^{n-1}_{ij})}{2}
        \\&+(R1)^{n}_{ij} +(R2)^{n}_{ij}+(R3)^{n}_{ij}+(R4)^{n}_{ij},\qquad(x_i,y_j)\in\Omega_{h},\quad 2\leq n\leq \mathcal{N},
   \end{split}
   \end{equation}
   \begin{equation}\label{eq3.3}
     u_{ij}^{n}=0,\qquad (x_i,y_j)\in\partial\Omega_{h},\qquad 1\leq n\leq \mathcal{N},
   \end{equation}
   \begin{equation}\label{eq3.4}
     u_{ij}^{0}=\psi(x_i,y_j),\qquad (x_i.y_j)\in\Omega_h,
   \end{equation}
   where
   \begin{equation*}
     \begin{split}
        &b^{n}_{ij}=\frac{1}{\tau}\int\limits^{t_{n}}_{t_{n-1}}f(x_i,y_j,t)dt,\\
        &(R3)^{n}_{ij}= \frac{1}{\tau}\int^{t_n}_{t_{n-1}}g(u(x_i,y_j,t))dt- \frac{g(u^n_{ij})+g(u^{n-1}_{ij})}{2}=\mathcal{O}(\tau^2),
     \end{split}
   \end{equation*}
   \begin{equation*}
     \begin{split}
        (R4)^{n}_{ij}&=\mathfrak{L}_{1,\tau}^{n}(\Delta u_{ij}^{n}-\Delta_{h}u_{ij}^{n})+\mathfrak{L}_{2,\tau}^{n}(\Delta u_{ij}^{n}-\Delta_{h}u_{ij}^{n})=\mathcal{O}(h_1^2+h_2^2).
     \end{split}
   \end{equation*}
   \vskip 0.2mm
   Omitting the truncation errors $(Rs)_{ij}^{n}(s=1,2,3,4)$, $1\leq n\leq \mathcal{N}$, and replacing $u_{ij}^n$ with $U_{ij}^{n}$, we obtain the following SCN finite difference scheme
   \begin{equation}\label{eq3.5}
        \delta_{t}U^{1}_{ij}-\mu\Delta_{h} U^{1}_{ij}- \mathfrak{w}_{1,1}\Delta_{h} U^{1}_{ij} =b^{1}_{ij}+\frac{g(U^1_{ij})+g(U^{0}_{ij})}{2},\qquad (x_i,y_j)\in\Omega_{h},
   \end{equation}
   \begin{equation}\label{eq3.6}
   \begin{split}
      \delta_{t}U^{n}_{ij}-\mu\Delta_{h} U^{n-\frac{1}{2}}_{ij}- \mathfrak{w}_{n,1}\Delta_{h} U^{1}_{ij}-\sum_{m=2}^{n}\mathfrak{w}_{n,m}\Delta_{h} U^{m-\frac{1}{2}}_{ij} =b^{n}_{ij}+\frac{g(U^n_{ij})+g(U^{n-1}_{ij})}{2},\\
        \qquad(x_i,y_j)\in\Omega_{h},\qquad 2\leq n\leq \mathcal{N},
   \end{split}
   \end{equation}
   \begin{equation}\label{eq3.7}
     U_{ij}^{n}=0,\qquad (x_i,y_j)\in\partial\Omega_{h},\qquad 1\leq n\leq \mathcal{N},
   \end{equation}
   \begin{equation}\label{eq3.8}
     U_{ij}^{0}=\psi(x_i,y_j),\qquad (x_i,y_j)\in\Omega_h.
   \end{equation}
   \vskip 0.2mm
   In order to solve \eqref{eq3.5}-\eqref{eq3.8} efficiently, we develop the following TTGCN finite difference method, which is divided into three steps.
   \vskip 0.2mm
\begin{itemize}
    \item   [\textbf{Step I.}]
 On the coarse grid, we only calculate $ks$-th level, $0\leq s\leq N$. Similar to the establishment of equations (\ref{eq3.5})-(\ref{eq3.6}), the discrete scheme on the coarse grid is constructed as follows
   \begin{equation}\label{eq3.9}
     \delta_{t}(U_C)^{k}_{ij}-\mu\Delta_{h} (U_C)^{k}_{ij}- \mathfrak{w}_{1,1}\Delta_{h} (U_C)^{k}_{ij} =b^{k}_{ij}+\frac{g((U_C)^k_{ij})+g((U_C)^{0}_{ij})}{2},\quad (x_i,y_j)\in\Omega_{h},
   \end{equation}
   \begin{equation}\label{eq3.10}
   \begin{split}
      \delta_{t}(U_C)^{sk}_{ij}&-\mu\Delta_{h} (U_C)^{(s-\frac{1}{2})k}_{ij}- \mathfrak{w}_{s,1}\Delta_{h} (U_C)^{k}_{ij}-\sum_{p=2}^{s}\mathfrak{w}_{s,p}\Delta_{h} (U_C)^{(p-\frac{1}{2})k}_{ij} \\&=b^{sk}_{ij}+\frac{g((U_C)^{sk}_{ij})+g((U_C)^{(s-1)k}_{ij})}{2},
        \qquad(x_i,y_j)\in\Omega_{h},\qquad 2\leq s\leq N.
   \end{split}
   \end{equation}
 \item  [\textbf{Step II.}] Then, based on the solution $(U_C)^{sk}_{ij}$ obtained in the Step I, applying Lagrange linear interpolation to calculate $(U_C)^{(s-1)k+q}_{ij}$ by points $(t_{(s-1)k},(U_C)^{(s-1)k}_{ij})$ and $(t_{sk},(U_C)^{sk}_{ij})$ direction on the coarse grid, with $1\leq q\leq k-1$, we have
   \begin{equation}\label{eq3.11}
     \begin{split}
        &\mathcal{L}_{U_C}(t_{(s-1)k+q})=U_{C}^{(s-1)k+q} \\
          & = \frac{t_{(s-1)k+q} - t_{sk}}{t_{(s-1)k} - t_{sk}} U_{C}^{(s-1)k}
        + \frac{t_{(s-1)k+q} - t_{(s-1)k}}{t_{sk} - t_{(s-1)k}} U_{C}^{sk} \\
          & = (1 - \frac{q}{k}) U_{C}^{(s-1)k} + \frac{q}{k} U_{C}^{sk},
        \quad 1\leq s \leq N, \quad 1\leq q \leq k-1.
     \end{split}
   \end{equation}
  \item [\textbf{Step III.}] Finally, according to $(U_C)^n_{ij}$ obtained in the Step II, the linear Crank-Nicolson finite difference scheme on a time fine grid is obtained by
   \begin{equation}\label{eq3.12}
   \begin{split}
      &\delta_{t}(U_F)^{1}_{ij}-\mu\Delta_{h} (U_F)^{1}_{ij}- \mathfrak{w}_{1,1}\Delta_{h} (U_F)^{1}_{ij}  \\
      &=b^{1}_{ij}+\frac{1}{2}g((U_F)^{0}_{ij})+ \frac{1}{2}\Big[g((U_C)^1_{ij})+g'((U_C)^1_{ij})\left((U_F)_{ij}^1-(U_C)_{ij}^1\right)\Big],\\ &(x_i,y_j)\in\Omega_{h},
   \end{split}
   \end{equation}
   \begin{equation}\label{eq3.13}
   \begin{split}
      &\delta_{t}(U_F)^{n}_{ij}-\mu\Delta_{h} (U_F)^{n-\frac{1}{2}}_{ij}- \mathfrak{w}_{n,1}\Delta_{h} (U_F)^{1}_{ij}-\sum_{p=2}^{n}\mathfrak{w}_{n,p}\Delta_{h} (U_F)^{p-\frac{1}{2}}_{ij} \\
      &=b^{n}_{ij}+\frac{1}{2}g((U_F)^{n-1}_{ij})+ \frac{1}{2}\Big[g((U_C)^n_{ij})+g'((U_C)^n_{ij}) \left((U_F)_{ij}^n-(U_C)_{ij}^n\right)\Big],\\
        &\qquad(x_i,y_j)\in\Omega_{h},\qquad 2\leq n\leq \mathcal{N}.
   \end{split}
   \end{equation}
   \end{itemize}
  \section{Analysis of the two-grid difference scheme}\label{secc4}
    \vskip 0.2mm
    Next, based on the TTGCN finite difference scheme \eqref{eq3.9}-\eqref{eq3.13}, we will analyze the stability and convergence of the scheme under the regularity assumption \textbf{(A1)} and \textbf{(A2)}.
    \vskip 0.2mm
    \subsection{Stability}
    We use the energy method to establish the stability of the TTGCN finite difference scheme. First, consider the case on the coarse grid.
    \begin{theorem}\label{th4.1}
      The fully discrete scheme \eqref{eq3.9}-\eqref{eq3.11} on the coarse grid is stable.
    \end{theorem}
    \begin{proof}
      Let $(\tilde{U}_C)_{ij}^{sk}$ be the approximation solution of \eqref{eq3.9}-\eqref{eq3.10}. Thus, we get
      \begin{equation}\label{eq4.1}
\begin{split}
        \delta_{t}(\tilde{U}_C)^{k}_{ij}-\mu\Delta_{h} (\tilde{U}_C)^{k}_{ij}&- \mathfrak{w}_{1,1}\Delta_{h} (\tilde{U}_C)^{k}_{ij}  =b^{k}_{ij}+\frac{g((\tilde{U}_C)^k_{ij})+g((\tilde{U}_C)^{0}_{ij})}{2},\\
   &(x_i,y_j)\in\Omega_{h},
\end{split}
   \end{equation}
   \begin{equation}\label{eq4.2}
   \begin{split}
   &   \delta_{t}(\tilde{U}_C)^{sk}_{ij}-\mu\Delta_{h} (\tilde{U}_C)^{(s-\frac{1}{2})k}_{ij}- \mathfrak{w}_{s,1}\Delta_{h} (\tilde{U}_C)^{k}_{ij}-\sum_{p=2}^{s}\mathfrak{w}_{s,p}\Delta_{h} (\tilde{U}_C)^{(p-\frac{1}{2})k}_{ij} \\&=b^{sk}_{ij}+\frac{g((\tilde{U}_C)^{sk}_{ij})+g((\tilde{U}_C)^{(s-1)k}_{ij})}{2},
        \quad(x_i,y_j)\in\Omega_{h},\quad 2\leq s\leq N.
   \end{split}
   \end{equation}
   Subtracting \eqref{eq4.1}-\eqref{eq4.2} from \eqref{eq3.9}-\eqref{eq3.10} and defining $\varepsilon_C=(U_C)^{sk}_{ij}-(\tilde U_C)^{sk}_{ij}$, we get
    \begin{equation}\label{eq4.3}
     \begin{split}
     &\delta_{t}(\varepsilon_C)^{k}_{ij}-\mu\Delta_{h} (\varepsilon_C)^{k}_{ij}- \mathfrak{w}_{1,1}\Delta_{h} (\varepsilon_C)^{k}_{ij} =\frac{1}{2}\left[g((U_C)^k_{ij})-g((\tilde{U}_C)^k_{ij})\right]\\ &+\frac{1}{2}\left[g((U_C)^0_{ij})-g((\tilde{U}_C)^0_{ij})\right] ,\quad (x_i,y_j)\in\Omega_{h},
     \end{split}
   \end{equation}
   \begin{equation}\label{eq4.4}
   \begin{split}
      &\delta_{t}(\varepsilon_C)^{sk}_{ij}-\mu\Delta_{h} (\varepsilon_C)^{(s-\frac{1}{2})k}_{ij}- \mathfrak{w}_{s,1}\Delta_{h} (\varepsilon_C)^{k}_{ij}-\sum_{p=2}^{s}\mathfrak{w}_{s,p}\Delta_{h} (\varepsilon_C)^{(p-\frac{1}{2})k}_{ij} \\ &=\frac{1}{2}\left[g((U_C)^{sk}_{ij})-g((\tilde{U}_C)^{sk}_{ij})\right] +\frac{1}{2}\left[g((U_C)^{(s-1)k}_{ij})-g((\tilde{U}_C)^{(s-1)k}_{ij})\right]
      ,\\
        &\qquad(x_i,y_j)\in\Omega_{h},\qquad 2\leq s\leq N.
   \end{split}
   \end{equation}
    \vskip 0.2mm
    We will prove this theorem in two steps as follows:
\begin{itemize}
    \item [(I)]
 Taking inner product of both sides of (\ref{eq4.3}) with $\varepsilon_C^k$ and multiplying it by $\tau_C$, we yield
    \begin{equation}\label{eq4.5}
     \begin{split}
     &\tau_C\left(\delta_{t}\varepsilon_C^{k},\varepsilon_C^k\right) -\tau_C\mu\left(\Delta_{h} \varepsilon_C^{k},\varepsilon_C^k\right)- \tau_C\mathfrak{w}_{1,1}\left(\Delta_{h} \varepsilon_C^{k},\varepsilon_C^k\right) \\ &=\frac{\tau_C}{2}\left(g(U_C^k)-g(\tilde{U}_C^k),\varepsilon_C^k\right) +\frac{\tau_C}{2}\left(g(U_C^0)-g(\tilde{U}_C^0),\varepsilon_C^k\right) ,\quad (x_i,y_j)\in\Omega_{h}.
     \end{split}
   \end{equation}
   \vskip 0.2mm
   For \eqref{eq4.4}, taking the inner product of both sides with $\varepsilon_C^{(s-\frac{1}{2})k}$, multiplying it by $\tau_C$, and summing for $s$ from 2 to $N$, we obtain
   \begin{equation}\label{eq4.6}
   \begin{split}
      &\sum^{N}_{s=2}\tau_C\left(\delta_{t}\varepsilon_C^{sk} ,\varepsilon_C^{(s-\frac{1}{2})k}\right) -\sum^{N}_{s=2}\mu\tau_C\left(\Delta_{h} \varepsilon_C^{(s-\frac{1}{2})k}, \varepsilon_C^{(s-\frac{1}{2})k}\right) -\sum^{N}_{s=2}\mathfrak{w}_{s,1}\tau_C \left(\Delta_{h}\varepsilon_C^{k} ,\varepsilon_C^{(s-\frac{1}{2})k}\right) \\&-\sum^{N}_{s=2}\tau_C\sum_{p=2}^{s}\mathfrak{w}_{s,p}\left(\Delta_{h} \varepsilon_C^{(p-\frac{1}{2})k},\varepsilon_C^{(s-\frac{1}{2})k}\right) =\sum^{N}_{s=2}\frac{\tau_C}{2}\left(g(U_C^{sk})-g(\tilde{U}_C^{sk}) ,\varepsilon_C^{(s-\frac{1}{2})k}\right) \\ &+\sum^{N}_{s=2}\frac{\tau_C}{2}\left(g(U_C^{(s-1)k})-g(\tilde{U}_C^{(s-1)k}) ,\varepsilon_C^{(s-\frac{1}{2})k}\right)
      ,\qquad(x_i,y_j)\in\Omega_{h},\qquad 2\leq s\leq N.
   \end{split}
   \end{equation}
   \vskip 0.2mm
   Then adding the above two equations together gives
   \begin{equation}\label{eq4.7}
   \begin{split}
      &\mathcal{H}_1+\mathcal{H}_2+\mathcal{H}_3 =\frac{\tau_C}{2}\left(g(U_C^k)-g(\tilde{U}_C^k),\varepsilon_C^k\right) +\frac{\tau_C}{2}\left(g(U_C^0)-g(\tilde{U}_C^0),\varepsilon_C^k\right) \\&+\sum^{N}_{s=2}\frac{\tau_C}{2}\left(g(U_C^{sk})-g(\tilde{U}_C^{sk}) ,\varepsilon_C^{(s-\frac{1}{2})k}\right)
     +\sum^{N}_{s=2}\frac{\tau_C}{2}\left(g(U_C^{(s-1)k})-g(\tilde{U}_C^{(s-1)k}) ,\varepsilon_C^{(s-\frac{1}{2})k}\right),
   \end{split}
   \end{equation}
   where
   \begin{equation*}
      \mathcal{H}_1=\tau_C\left(\delta_{t}\varepsilon_C^{k},\varepsilon_C^k\right) +\sum^{N}_{s=2}\tau_C\left(\delta_{t}\varepsilon_C^{sk} ,\varepsilon_C^{(s-\frac{1}{2})k}\right),
      \end{equation*}
      \begin{equation*}
      \mathcal{H}_2=-\tau_C\mu\left(\Delta_{h} \varepsilon_C^{k},\varepsilon_C^k\right)-\sum^{N}_{s=2}\mu\tau_C\left(\Delta_{h} \varepsilon_C^{(s-\frac{1}{2})k}, \varepsilon_C^{(s-\frac{1}{2})k}\right),
      \end{equation*}
      \begin{equation*}
      \begin{split}
      \mathcal{H}_3 =&-\tau_C\mathfrak{w}_{1,1}\left(\Delta_{h} \varepsilon_C^{k},\varepsilon_C^k\right)-\sum^{N}_{s=2}\mathfrak{w}_{s,1}\tau_C \left(\Delta_{h}\varepsilon_C^{k} ,\varepsilon_C^{(s-\frac{1}{2})k}\right) \\
      &-\sum^{N}_{s=2}\tau_C\sum_{p=2}^{s}\mathfrak{w}_{s,p}\left(\Delta_{h} \varepsilon_C^{(p-\frac{1}{2})k},\varepsilon_C^{(s-\frac{1}{2})k}\right).
      \end{split}
      \end{equation*}
      \vskip 0.2mm
      Below the terms $\mathcal{H}_{q}(q=1,2,3)$ will be estimated one by one. First, for $\mathcal{H}_{1}$, we use Lemma \ref{lemma2.7} to obtain
      \begin{equation}\label{eq4.8}
        \mathcal{H}_1\geq \frac{1}{2}\left(\|\varepsilon_C^{Nk}\|^{2}-\|\varepsilon_C^{0}\|^{2}\right).
      \end{equation}
      \vskip 0.2mm
      Second, from Lemma \ref{lemma2.5}, we obtain
      \begin{equation}\label{eq4.9}
        \begin{split}
           \mathcal{H}_2&= \mu\tau_C\left(\nabla_{h} \varepsilon_C^{k},\nabla_{h}\varepsilon_C^k\right)+\sum^{N}_{s=2}\mu\tau_C\left(\nabla_{h} \varepsilon_C^{(s-\frac{1}{2})k}, \nabla_{h}\varepsilon_C^{(s-\frac{1}{2})k}\right) \\
             &=\mu\tau_C\|\nabla_{h}\varepsilon_C^{k}\|^{2}+ \sum^{N}_{s=2}\mu\tau_C\|\nabla_{h}\varepsilon_C^{(s-\frac{1}{2})k}\|^{2}\geq 0.
        \end{split}
      \end{equation}

      \vskip 0.2mm
      Finally, for the third term $\mathcal{H}_3$, we use Lemma \ref{lemma2.5} and Lemma \ref{lemma2.6} to get
      \begin{equation}\label{eq4.10}
        \begin{split}
           \mathcal{H}_3&=\tau_C\mathfrak{w}_{1,1}\left(\nabla_{h} \varepsilon_C^{k},\nabla_{h}\varepsilon_C^k\right)+\sum^{N}_{s=2}\mathfrak{w}_{s,1}\tau_C \left(\nabla_{h}\varepsilon_C^{k} ,\nabla_{h}\varepsilon_C^{(s-\frac{1}{2})k}\right) \\
      &+\sum^{N}_{s=2}\tau_C\sum_{p=2}^{s}\mathfrak{w}_{s,p}\left(\nabla_{h} \varepsilon_C^{(p-\frac{1}{2})k},\nabla_{h}\varepsilon_C^{(s-\frac{1}{2})k}\right) \\
             & =\tau_C\mathfrak{w}_{1,1}\left(\nabla_{h} \varepsilon_C^{k},\nabla_{h}\varepsilon_C^k\right)+\sum^{N}_{s=2}\tau_C \left(\mathfrak{w}_{s,1}\nabla_{h}\varepsilon_C^{k} +\sum_{p=2}^{s}\mathfrak{w}_{s,p}\nabla_{h} \varepsilon_C^{(p-\frac{1}{2})k},\nabla_{h}\varepsilon_C^{(s-\frac{1}{2})k}\right) \\
             & =\tau_C\mathfrak{w}_{1,1}\left(\nabla_{h} \varepsilon_C^{k},\nabla_{h}\varepsilon_C^k\right) +\sum^{N}_{s=2}\tau_C(\mathfrak{L}_{2,\tau}^{s}\nabla_{h}\varepsilon_C^{sk} ,\nabla_{h}\varepsilon_C^{(s-\frac{1}{2})k})\geq 0.
        \end{split}
      \end{equation}
      \vskip 0.2mm
      Next, $g(u)$ satisfies the Lipschitz condition. For (\ref{eq4.7}), using Cauchy-Schwarz inequality, we have
      \begin{equation}\label{eq4.11}
        \begin{split}
           &\|\varepsilon_C^{Nk}\|^{2}-\|\varepsilon_C^{0}\|^{2}\\ &\leq\tau_C\left\|g(U_C^k)-g(\tilde{U}_C^k)\right\| \left\|\varepsilon_C^k\right\| +\tau_C\left\|g(U_C^0)-g(\tilde{U}_C^0)\right\|\left\|\varepsilon_C^k\right\| \\&+\sum^{N}_{s=2}\tau_C\left\|g(U_C^{sk})-g(\tilde{U}_C^{sk}) \right\|\left\|\varepsilon_C^{(s-\frac{1}{2})k}\right\|
        +\sum^{N}_{s=2}\tau_C\left\|g(U_C^{(s-1)k})-g(\tilde{U}_C^{(s-1)k}) \right\|\left\|\varepsilon_C^{(s-\frac{1}{2})k}\right\|  \\
             &\leq \bar{C}\tau_C\left(\left\|\varepsilon_C^k\right\|^2+ \left\|\varepsilon_C^0\right\|\left\|\varepsilon_C^k\right\| +\sum^{N}_{s=2}\left\|\varepsilon_C^{sk}\right\| \left\|\varepsilon_C^{(s-\frac{1}{2})k}\right\| +\sum^{N}_{s=2}\left\|\varepsilon_C^{(s-1)k}\right\| \left\|\varepsilon_C^{(s-\frac{1}{2})k}\right\|\right).
        \end{split}
      \end{equation}
      \vskip 0.2mm
      Now, taking the positive integer $\bar{m}$ such that $\left\|\varepsilon_C^{\bar{m}k}\right\|=\max\limits_{0\leq s\leq N}\left\|\varepsilon_C^{sk}\right\|$, we have
      \begin{equation}\label{eq4.12}
        \begin{split}
          \|\varepsilon_C^{Nk}\|\leq \|\varepsilon_C^{\bar{m}k}\| & \leq \|\varepsilon_C^{0}\|+\bar{C}\tau_C\left(\left\|\varepsilon_C^k\right\|+ \left\|\varepsilon_C^0\right\| +\sum^{\bar{m}}_{s=2}\left\|\varepsilon_C^{sk}\right\| +\sum^{\bar{m}}_{s=2}\left\|\varepsilon_C^{(s-1)k}\right\| \right)\\
             &\leq\|\varepsilon_C^{0}\|+\bar{C}\tau_C\left(\left\|\varepsilon_C^k\right\|+ \left\|\varepsilon_C^0\right\| +\sum^{N}_{s=2}\left\|\varepsilon_C^{sk}\right\| +\sum^{N}_{s=2}\left\|\varepsilon_C^{(s-1)k}\right\| \right)\\
             &\leq\|\varepsilon_C^{0}\|+\bar{C}\tau_C \left( \sum^{N}_{s=0}\left\|\varepsilon_C^{sk}\right\| +\sum^{N-1}_{s=1}\left\|\varepsilon_C^{sk}\right\| \right)\\
             &\leq\|\varepsilon_C^{0}\|+\bar{C}\tau_C\left\|\varepsilon_C^{Nk}\right\| +2\bar{C}\tau_C \sum^{N-1}_{s=0}\left\|\varepsilon_C^{sk}\right\|.
        \end{split}
      \end{equation}
      \vskip 0.2mm
      When $\tau_C\leq \frac{1}{2\bar{C}}$, following from Lemma \ref{lemma2.8}, inequality (\ref{eq4.12}) becomes
      \begin{equation}\label{eq4.13}
        \begin{split}
            \|\varepsilon_C^{Nk}\| &\leq \bar{C}(T)\|\varepsilon_{C}^{0}\|\exp\{N\tau_C\}\leq \bar{C}\|\varepsilon_{C}^{0}\|.
        \end{split}
      \end{equation}
      \vskip 0.2mm
       \item [(II)] Notice that according to \textbf{(I)} we have $\|U_C^{sk}\|\leq \bar{C}$ for any $1\leq s\leq N$. Then we estimate the $\|U_{C}^{(s-1)k+q}\|$ for $1\leq s\leq N$ and $1\leq q\leq k-1$. Considering \eqref{eq3.11} and applying the triangle inequality, we obtain
       \begin{equation}\label{eq4.14}
         \begin{split}
            \|U_{C}^{(s-1)k+q}\|=\|(1-\frac{q}{k}) U_C^{(s-1)k}+\frac{q}{k}U_C^{sk}\|
            \leq (1-\frac{q}{k})\|U_C^{(s-1)k}\|+\frac{q}{k}\|U_C^{sk}\|
            \leq \bar{C},
         \end{split}
       \end{equation}
       \end{itemize}
    which completes the proof.
    \end{proof}
    \vskip 0.2mm
    In addition, we shall analyse the stability on the fine grid.
    \begin{theorem}
      For the system \eqref{eq3.12} and \eqref{eq3.13} on the fine grid, with $1\leq n\leq \mathcal{N}$, we have $\|U_F^n\|\leq \bar{C}$.
    \end{theorem}
    \begin{proof}
     \vskip 0.2mm
      Taking the inner product of \eqref{eq3.12} with $\tau_FU_F^1$, we have
      \begin{equation}\label{eq4.15}
        \begin{split}
          &\tau_F\left(\delta_{t}U_F^{1},U_F^1\right)-\mu\tau_F\left(\Delta_{h} U_F^{1},U_F^1\right)- \tau_F\mathfrak{w}_{1,1}\left(\Delta_{h} U_F^{1},U_F^1\right)  \\
      &=\tau_F\left(b^{1},U_F^1\right)+\frac{\tau_F}{2}\left(g(U_F^{0}),U_F^1\right)+ \frac{\tau_F}{2}\left(g(U_C^1)+ g'(U_C^1)\left(U_F^1-U_C^1\right),U_F^1\right).
        \end{split}
      \end{equation}
      \vskip 0.2mm
   For \eqref{eq3.13}, taking the inner product of both sides with $U_F^{n-\frac{1}{2}}$, multiplying it by $\tau_F$, and summing for $n$ from 2 to $\mathcal{N}$, we get
    \begin{equation}\label{eq4.16}
   \begin{split}
      &\sum^{\mathcal{N}}_{n=2}\tau_F\left(\delta_{t}U_F^{n},U_F^{n-\frac{1}{2}}\right) -\sum^{\mathcal{N}}_{n=2}\tau_F\mu\left(\Delta_{h} U_F^{n-\frac{1}{2}},U_F^{n-\frac{1}{2}}\right)- \sum^{\mathcal{N}}_{n=2}\tau_F\mathfrak{w}_{n,1}\left(\Delta_{h} U_F^{1},U_F^{n-\frac{1}{2}}\right)\\
      &-\sum^{\mathcal{N}}_{n=2}\tau_F\sum_{p=2}^{n}\mathfrak{w}_{n,p}\left(\Delta_{h} U_F^{p-\frac{1}{2}},U_F^{n-\frac{1}{2}}\right) =\sum^{\mathcal{N}}_{n=2}\tau_F\left(b^{n},U_F^{n-\frac{1}{2}}\right) +\sum^{\mathcal{N}}_{n=2}\frac{\tau_F}{2}\left(g(U_F^{n-1}),U_F^{n-\frac{1}{2}}\right)\\
      &+\sum^{\mathcal{N}}_{n=2}\frac{\tau_F}{2} \left(g(U_C^n)+g'(U_C^n) \left(U_F^n-U_C^n\right),U_F^{n-\frac{1}{2}}\right),
        \qquad(x_i,y_j)\in\Omega_{h},\qquad 2\leq n\leq \mathcal{N}.
   \end{split}
   \end{equation}
   \vskip 0.2mm
   Then, adding (\ref{eq4.15}) and (\ref{eq4.16}), and similar to the analysis of (\ref{eq4.6})-(\ref{eq4.10}), we obtain
   \begin{equation}\label{eq4.17}
     \begin{split}
        &\|U_F^{\mathcal{N}}\|^2-\|U_F^{0}\|^2\\
        &\leq  2\tau_F\|b^{1}\|\|U_F^1\|+ 2\sum^{\mathcal{N}}_{n=2}\tau_F\|b^{n}\|\|U_F^{n-\frac{1}{2}}\| +\tau_F\|g(U_F^{0})\|\|U_F^1\| +\sum^{\mathcal{N}}_{n=2}\tau_F\|g(U_F^{n-1})\|\|U_F^{n-\frac{1}{2}}\| \\
        & +\tau_F\|g(U_C^1)\|\|U_F^1\|+ \tau_F\|g'(U_C^1)\left(U_F^1-U_C^1\right)\|\|U_F^1\| \\
        &+\sum^{\mathcal{N}}_{n=2}\tau_F \|g(U_C^n)\|\|U_F^{n-\frac{1}{2}}\| +\sum^{\mathcal{N}}_{n=2}\tau_F\|g'(U_C^n) \left(U_F^n-U_C^n\right)\|\|U_F^{n-\frac{1}{2}}\|.
     \end{split}
   \end{equation}
   \vskip 0.2mm
   Based on the stability of the coarse grid, $\|U_C^n\|\leq \bar{C}(0\leq n\leq \mathcal{N})$ can be obtained. Then according to $g(u)\in C^2(\mathbf{R})\cap L^1(0,T]$, we have $g(U_C^n)\leq \bar{C}$ and $g'(U_C^n)\leq \bar{C}$. Also, assuming $\|U_F^n\|\leq \bar{C}$ holds for $0\leq n\leq \mathcal{N}-1$, then $g(U_F^n)\leq \bar{C}$ can be obtained, thus
   \begin{equation}\label{eq4.18}
     \begin{split}
        \|U_F^{\mathcal{N}}\|^2-\|U_F^{0}\|^2 & \leq 2\tau_F\|b^1\|\|U_F^1\| +2\sum^{\mathcal{N}}_{n=2}\tau_F\|b^{n}\|\|U_F^{n-\frac{1}{2}}\| +\bar{C}\tau_F\|U_F^1\|\\
        &+\bar{C}\sum^{\mathcal{N}}_{n=2}\tau_F\|U_F^{n-\frac{1}{2}}\| +\bar{C}\tau_F\left(\|U_F^1\|+\|U_C^1\|\right)\|U_F^1\| \\ & +\bar{C}\sum^{\mathcal{N}}_{n=2}\tau_F\left(\|U_F^n\|+\|U_C^n\|\right)\|U_F^{n-\frac{1}{2}}\|.
     \end{split}
   \end{equation}
   \vskip 0.2mm
   Denoting $\|U_F^{\check{m}}\|=\max\limits_{0\leq n\leq \mathcal{N}}\|U_F^n\|$, we can get
   \begin{equation}\label{eq4.19}
     \begin{split}
        \|U_F^{\tilde{m}}\|^2&\leq \|U_F^{0}\|^2+2\tau_F\|b^1\|\|U_F^1\| +2\sum^{\tilde{m}}_{n=2}\tau_F\|b^{n}\|\|U_F^{n-\frac{1}{2}}\| +\bar{C}\tau_F\|U_F^1\|
        +\bar{C}\sum^{\tilde{m}}_{n=2}\tau_F\|U_F^{n-\frac{1}{2}}\|\\
        &+\bar{C}\tau_F\left(\|U_F^1\|+\|U_C^1\|\right)\|U_F^1\| +\bar{C}\sum^{\tilde{m}}_{n=2}\tau_F\left(\|U_F^n\|+\|U_C^n\|\right)\|U_F^{n-\frac{1}{2}}\|\\
        &\leq \|U_F^{0}\| \|U_F^{\tilde{m}}\|+2\tau_F\|b^1\| \|U_F^{\tilde{m}}\| +2\sum^{\tilde{m}}_{n=2}\tau_F\|b^{n}\| \|U_F^{\tilde{m}}\| +\bar{C}\tau_F \|U_F^{\tilde{m}}\|
        +\bar{C}\sum^{\tilde{m}}_{n=2}\tau_F \|U_F^{\tilde{m}}\|\\
        &+\bar{C}\tau_F\left(\|U_F^1\|+\|U_C^1\|\right) \|U_F^{\tilde{m}}\| +\bar{C}\sum^{\tilde{m}}_{n=2}\tau_F\left(\|U_F^n\|+\|U_C^n\|\right) \|U_F^{\tilde{m}}\|.
     \end{split}
   \end{equation}
   \vskip 0.2mm
   Then
   \begin{equation}\label{eq4.20}
     \begin{split}
        \|U_F^{\mathcal{N}}\|\leq\|U_F^{\tilde{m}}\| &\leq \|U_F^{0}\| +2\sum^{\tilde{m}}_{n=1}\tau_F\|b^{n}\|+\bar{C}\sum^{\tilde{m}}_{n=1}\tau_F +\bar{C}\sum^{\tilde{m}}_{n=1}\tau_F\left(\|U_F^n\|+\|U_C^n\|\right)\\
        &\leq \|U_F^{0}\| +2\sum^{\mathcal{N}}_{n=1}\tau_F\|b^{n}\|+\bar{C}\sum^{\mathcal{N}}_{n=1}\tau_F +\bar{C}\sum^{\mathcal{N}}_{n=1}\tau_F\left(\|U_F^n\|+\|U_C^n\|\right).
     \end{split}
   \end{equation}
   When $\tau_F\leq \frac{1}{4\bar{C}}$, from Lemma \ref{lemma2.8} and Theorem \ref{th4.1}, inequality (\ref{eq4.20}) turn into the following
   \begin{equation}\label{eq4.21}
     \begin{split}
       \|U_F^{\mathcal{N}}\|\leq \bar{C}(T)\exp(\mathcal{N}\tau_F)\left(\|U_F^{0}\| +\sum^{\mathcal{N}}_{n=1}\tau_F\|b^{n}\|+\sum^{\mathcal{N}}_{n=1}\tau_F +\sum^{\mathcal{N}}_{n=1}\tau_F\|U_C^n\|\right)\leq \bar{C}.
     \end{split}
   \end{equation}
   This finishes the proof.
    \end{proof}
    \subsection{Convergence}
    The convergence of TTGCN finite difference scheme \eqref{eq3.9}-\eqref{eq3.11} on coarse grid will be analysis using the energy method. Let
    \begin{equation*}
      (e_C)_{ij}^{n}=u_{ij}^{n}-(U_C)_{ij}^{n},\qquad (x_i,y_j)\in\bar{\Omega}_{h},\qquad 0\leq n\leq \mathcal{N}.
    \end{equation*}
    \vskip 0.2mm
    Subtracting (\ref{eq3.9})-(\ref{eq3.10}), (\ref{eq3.7})-(\ref{eq3.8}) from (\ref{eq3.1})-(\ref{eq3.4}), respectively, we obtain the following error equations
     \begin{equation}\label{eq4.22}
     \begin{split}
     &\delta_{t}(e_C)^{k}_{ij}-\mu\Delta_{h} (e_C)^{k}_{ij}- \mathfrak{w}_{1,1}\Delta_{h} (e_C)^{k}_{ij} =\frac{1}{2}\left[g(u^k_{ij})-g((U_C)^k_{ij})\right]\\ &+\frac{1}{2}\left[g(u^0_{ij})-g((U_C)^0_{ij})\right]+ (R)^{k}_{ij},\quad (x_i,y_j)\in\Omega_{h},
     \end{split}
   \end{equation}
   \begin{equation}\label{eq4.23}
   \begin{split}
      &\delta_{t}(e_C)^{sk}_{ij}-\mu\Delta_{h} (e_C)^{(s-\frac{1}{2})k}_{ij}- \mathfrak{w}_{s,1}\Delta_{h} (e_C)^{k}_{ij}-\sum_{p=2}^{s}\mathfrak{w}_{s,p}\Delta_{h} (e_C)^{(p-\frac{1}{2})k}_{ij} \\ &=\frac{1}{2}\left[g(u^{sk}_{ij})-g((U_C)^{sk}_{ij})\right] +\frac{1}{2}\left[g(u^{(s-1)k}_{ij})-g((U_C)^{(s-1)k}_{ij})\right]
      +(R)^{sk}_{ij},\\
        &\qquad(x_i,y_j)\in\Omega_{h},\qquad 2\leq s\leq N,
   \end{split}
   \end{equation}
    \begin{equation}\label{eq4.24}
     (e_C)_{ij}^{n}=0,\qquad (x_i,y_j)\in\partial\Omega_{h},\qquad 1\leq n\leq \mathcal{N},
   \end{equation}
   \begin{equation}\label{eq4.25}
     (e_C)_{ij}^{0}=0,\qquad (x_i,y_j)\in\Omega_h,
   \end{equation}
   where $(R)=(R1)+(R2)+(R3)+(R4)$.

   \begin{theorem}\label{th4.3}
     Assume that $u(x,y,t)$ and $U_C^n$ are solutions of \eqref{eq3.1}-\eqref{eq3.2} and \eqref{eq3.9}-\eqref{eq3.10}, respectively, and that $u(x,y,t)$ satisfies the regularity assumptions \textbf{(A1)} and \textbf{(A2)}. Then, it holds that
     \begin{equation*}
     \begin{array}{ll}
     	\max\limits_{1\leq n\leq \mathscr{N}} \|u^n - U_C^n\|  \leq  \bar{C}( \tau_C^{2} + h_1^2 + h_2^2 ),\quad 1\leq n \leq \mathscr{N}.
        \end{array}
     \end{equation*}
   \end{theorem}
   \begin{proof}
     The proof of this theorem is divided into two steps:
     \vskip 0.2mm
     \textbf{(I)}. Taking the inner product of equations (\ref{eq4.22}) and (\ref{eq4.23}) with $e_C^k$ and $e_C^{(s-\frac{1}{2})k}$ respectively, and multiplying both equations by $\tau_C$, summing for $s$ from 2 to $N$ in (\ref{eq4.23}) and adding (\ref{eq4.22}), then we can obtain
     \begin{equation}\label{eq4.26}
   \begin{split}
      &\tilde{\mathcal{H}}_1+\tilde{\mathcal{H}}_2+\tilde{\mathcal{H}}_3 =\frac{\tau_C}{2}\left(g(u^k)-g(U_C^k),e_C^k\right) +\frac{\tau_C}{2}\left(g(u^0)-g(U_C^0),e_C^k\right) \\&+\sum^{N}_{s=2}\frac{\tau_C}{2}\left(g(u^{sk})-g(U_C^{sk}) ,e_C^{(s-\frac{1}{2})k}\right)
     +\sum^{N}_{s=2}\frac{\tau_C}{2}\left(g(u^{(s-1)k})-g(U_C^{(s-1)k}) ,e_C^{(s-\frac{1}{2})k}\right)\\
     &+\tau_C\left((R)^k,e_C^k\right) +\sum_{s=2}^{N}\tau_C\left((R)^{sk},e_C^{(s-\frac{1}{2})k}\right),
   \end{split}
   \end{equation}
   where
   \begin{equation*}
      \tilde{\mathcal{H}}_1=\tau_C\left(\delta_{t}e_C^{k},e_C^k\right) +\sum^{N}_{s=2}\tau_C\left(\delta_{t}e_C^{sk} ,e_C^{(s-\frac{1}{2})k}\right),
      \end{equation*}
      \begin{equation*}
      \tilde{\mathcal{H}}_1=-\tau_C\mu\left(\Delta_{h} e_C^{k},e_C^k\right)-\sum^{N}_{s=2}\mu\tau_C\left(\Delta_{h} e_C^{(s-\frac{1}{2})k}, e_C^{(s-\frac{1}{2})k}\right),
      \end{equation*}
      \begin{equation*}
      \begin{split}
      \tilde{\mathcal{H}}_3 =&-\tau_C\mathfrak{w}_{1,1}\left(\Delta_{h} e_C^{k},e_C^k\right)-\sum^{N}_{s=2}\mathfrak{w}_{s,1}\tau_C \left(\Delta_{h}e_C^{k} ,e_C^{(s-\frac{1}{2})k}\right) \\
      &-\sum^{N}_{s=2}\tau_C\sum_{p=2}^{s}\mathfrak{w}_{s,p}\left(\Delta_{h} e_C^{(p-\frac{1}{2})k},e_C^{(s-\frac{1}{2})k}\right).
      \end{split}
      \end{equation*}
      \vskip 0.2mm
      For (\ref{eq4.26}), applying Lemmas \ref{lemma2.5}-\ref{lemma2.7} and Cauchy-Schwarz inequality, we get the following inequality
      \begin{equation}\label{eq4.27}
        \begin{split}
           &\|e_C^{Nk}\|^{2}-\|e_C^{0}\|^{2}\\ &\leq\tau_C\|g(u^k)-g(U_C^k)\| \|e_C^k\| +\tau_C\|g(u^0)-g(U_C^0)\|\|e_C^k\| \\&+\sum^{N}_{s=2}\tau_C\|g(u^{sk})-g(U_C^{sk}) \|\|e_C^{(s-\frac{1}{2})k}\|
        +\sum^{N}_{s=2}\tau_C\|g(u^{(s-1)k})-g(U_C^{(s-1)k}) \|\|e_C^{(s-\frac{1}{2})k}\|  \\
             &+2\tau_C\|(R)^k\|\|e_C^k\| +2\sum_{s=2}^{N}\tau_C\|(R)^{sk}\|\|e_C^{(s-\frac{1}{2})k}\|
             \\&\leq \bar{C}\tau_C\left(\|e_C^k\|^2+ \|e_C^0\|\|e_C^k\| +\sum^{N}_{s=2}\|e_C^{sk}\| \|e_C^{(s-\frac{1}{2})k}\| +\sum^{N}_{s=2}\|e_C^{(s-1)k}\| \|e_C^{(s-\frac{1}{2})k}\|\right)\\
             &+2\tau_C\|(R)^k\|\|e_C^k\| +2\sum_{s=2}^{N}\tau_C\|(R)^{sk}\|\|e_C^{(s-\frac{1}{2})k}\|.
        \end{split}
      \end{equation}
      \vskip 0.2mm
      Choosing a positive integer $\bar{s}$ such that $\|e_C^{\bar{s}k}\|= \max\limits_{0\leq s\leq N}\|e_C^{sk}\|$ and noting that (\ref{eq4.24}), then we have
      \begin{equation}\label{eq4.28}
        \begin{split}
        \|e_C^{Nk}\|\leq \|e_C^{\bar{s}k}\|&\leq \bar{C}\tau_C\left(\|e_C^k\|+ \sum^{\bar{s}}_{s=2}\|e_C^{sk}\| +\sum^{\bar{s}}_{s=2}\|e_C^{(s-1)k}\|\right)
        +2\sum_{s=1}^{\bar{s}}\tau_C\|(R)^{sk}\|\\
        &\leq \bar{C}\tau_C\left(\|e_C^k\|+ \sum^{N}_{s=2}\|e_C^{sk}\| +\sum^{N}_{s=2}\|e_C^{(s-1)k}\|\right)
        +2\sum_{s=1}^{N}\tau_C\|(R)^{sk}\|\\
        &\leq \bar{C}\tau_C\left(\sum^{N}_{s=1}\|e_C^{sk}\|
        +\sum_{s=1}^{N}\|(R)^{sk}\|\right).
        \end{split}
      \end{equation}
       \vskip 0.2mm
       Using Lemma \ref{lemma2.8}, then \eqref{eq4.28} becomes the following
       \begin{equation}\label{eq4.29}
         \|e_C^{Nk}\|\leq \bar{C}(T)\exp\{N\tau_C\}\left(\tau_C\sum^{N}_{s=1}\|(R)^{sk}\|\right).
       \end{equation}
       \vskip 0.2mm
       In addition, from Lemmas \ref{lemma2.1}-\ref{lemma2.4} and using triangle inequality, we can get the following estimates
       \begin{equation}\label{eq4.30}
       \begin{split}
          \tau_C\sum^{N}_{s=1}\|(R)^{sk}\|&= \tau_C\sum^{N}_{s=1}\|(R1)^{sk}+(R2)^{sk}+(R3)^{sk}+(R4)^{sk}\| \\
            &\leq \tau_C\sum^{N}_{s=1}\left(\|(R1)^{sk}\|+\|(R2)^{sk}\|+\|(R3)^{sk}\|+\|(R4)^{sk}\|\right)\\
            &\leq \bar{C}(T)(\tau_C^2+h_1^2+h_2^2).
       \end{split}
       \end{equation}
       \vskip 0.2mm
       Finally combining (\ref{eq4.29}) and (\ref{eq4.30}), we have
       \begin{equation}\label{eq4.31}
         \|e_C^{sk}\|\leq\bar{C}(T)(\tau_C^2+h_1^2+h_2^2),\qquad 1\leq s\leq N.
       \end{equation}
       \textbf{(II)}. For any $1\leq s \leq N$ and $1 \leq q \leq k-1$, we utilize the Lagrange's interpolation formula, then
    \begin{equation}\label{eq4.32}
    \begin{split}
       u^{(s-1)k+q}  &=  (1 - \frac{q}{k}) u^{(s-1)k} + \frac{q}{k} u^{sk}\\
       &+ \frac{u''(\xi)}{2}(t_{(s-1)k+q} - t_{(s-1)k} )(t_{(s-1)k+q} - t_{sk} ),\qquad \xi\in(t_{(s-1)k},t_{sk}).
    \end{split}
     \end{equation}

 \vskip 0.2mm
    Subtracting (\ref{eq3.11}) from (\ref{eq4.32}), we have
  \begin{equation*}
    \begin{split}
     e_C^{(s-1)k+q}
     = (1 - \frac{q}{k}) e_C^{(s-1)k} + \frac{q}{k} e_C^{sk}
       + \frac{u''(\xi)}{2}(t_{(s-1)k+q} - t_{(s-1)k} )(t_{(s-1)k+q} - t_{sk} ),
    \end{split}
    \end{equation*}
    then, applying the triangle inequality and (\ref{eq4.31}), we obtain
  \begin{equation}\label{eq4.33}
    \begin{split}
     \|e_C^{(s-1)k+q}\|
    &\leq (1 - \frac{q}{k}) \|e_C^{(s-1)k}\| + \frac{q}{k} \|e_C^{sk}\|
       + \frac{\|u''(\xi)\|_{\infty}}{2}\tau_C^2  \\
    &\leq \bar{C}(\tau_{C}^{2} + h_1^2 + h_2^2 ),  \qquad  1 \leq s \leq N,  \qquad  1 \leq q \leq k-1.
    \end{split}
  \end{equation}
    The proof is finished.
   \end{proof}
   \vskip 0.2mm
   Next, the convergence on the fine grid will be considered. Let
   \begin{equation*}
      (e_F)_{ij}^{n}=u_{ij}^{n}-(U_F)_{ij}^{n},\qquad (x_i,y_j)\in\bar{\Omega}_{h},\qquad 0\leq n\leq \mathcal{N}.
    \end{equation*}
    \vskip 0.2mm
    Subtracting (\ref{eq3.12})-(\ref{eq3.13}), (\ref{eq3.7})-(\ref{eq3.8}) from (\ref{eq3.1})-(\ref{eq3.4}), respectively, we yield the following error equations
     \begin{equation}\label{eq4.34}
     \begin{split}
     &\delta_{t}(e_F)^{1}_{ij}-\mu\Delta_{h} (e_F)^{1}_{ij}- \mathfrak{w}_{1,1}\Delta_{h} (e_F)^{1}_{ij} =\frac{1}{2}\left[g(u^0_{ij})-g((U_F)^0_{ij})\right]\\ &+\frac{1}{2}\left[g(u_{ij}^1)-g((U_C)^1_{ij})-g'((U_C)^1_{ij}) \left((U_F)_{ij}^1-(U_C)_{ij}^1\right)\right]+(R)_{ij}^1,\quad (x_i,y_j)\in\Omega_{h},
     \end{split}
   \end{equation}
   \begin{equation}\label{eq4.35}
   \begin{split}
      &\delta_{t}(e_F)^{n}_{ij}-\mu\Delta_{h} (e_F)^{n-\frac{1}{2}}_{ij}- \mathfrak{w}_{n,1}\Delta_{h} (e_F)^{1}_{ij}-\sum_{p=2}^{n}\mathfrak{w}_{n,p}\Delta_{h} (e_F)^{p-\frac{1}{2}}_{ij}  \\&=
\frac{1}{2}\left[g(u^{n-1}_{ij})
      -g((U_F)^{n-1}_{ij})\right] +\frac{1}{2}\left[g(u_{ij}^n)-g((U_C)^n_{ij})-g'((U_C)^n_{ij}) \left((U_F)_{ij}^n-(U_C)_{ij}^n\right)\right]
      \\
        &+(R)^{n}_{ij},\qquad(x_i,y_j)\in\Omega_{h},\qquad 2\leq n\leq \mathcal{N},
   \end{split}
   \end{equation}
    \begin{equation}\label{eq4.36}
     (e_F)_{ij}^{n}=0,\qquad (x_i,y_j)\in\partial\Omega_{h},\qquad 1\leq n\leq \mathcal{N},
   \end{equation}
   \begin{equation}\label{eq4.37}
     (e_F)_{ij}^{0}=0,\qquad (x_i,y_j)\in\Omega_h.
   \end{equation}
   \begin{theorem}
     Assume that $u(x,y,t)$ and $U_F^n$ are solutions of \eqref{eq3.1}-\eqref{eq3.2} and \eqref{eq3.12}-\eqref{eq3.13}, respectively, and let $u(x,y,t)$ satisfy the regularity assumption \textbf{(A1)} and \textbf{(A2)}, then we have the following
       \begin{equation*}
     \begin{split}
     	\|e_F^n\|  \leq  \bar{C}( \tau_F^{2} + \tau_C^{4} + h_1^2 + h_2^2 ),\qquad 1\leq n\leq \mathcal{N}.
        \end{split}
     \end{equation*}
     \begin{proof}
       Taking the inner product of (\ref{eq4.34}) with $\tau_F e_F^1$, we obtain
       \begin{equation}\label{eq4.38}
         \begin{split}
            &\tau_F\left(\delta_{t}e_F^{1},e_F^1\right)-\mu\tau_F\left(\Delta_{h} e_F^{1},e_F^1\right)- \mathfrak{w}_{1,1}\tau_F\left(\Delta_{h} e_F^{1},e_F^1\right) =\frac{\tau_F}{2}\left(g(u^0)-g(U_F^0),e_F^1\right)\\ &+\frac{\tau_F}{2}\left(g(u^1)-g(U_C^1)-g'(U_C^1) \left(U_F^1-U_C^1\right),e_F^1\right)+\tau_F\left((R)^1,e_F^1\right).
         \end{split}
       \end{equation}
       \vskip 0.2mm
       Then taking the inner product of equation (\ref{eq4.35}) with $\tau_F e_F^{n-\frac{1}{2}}$ and summing for $n$ from $2$ to $\mathcal{N}$, we can get
       \begin{equation}\label{eq4.39}
         \begin{split}
            &\sum^{\mathcal{N}}_{n=2}\tau_F\left(\delta_{t}e_F^{n},e_F^{n-\frac{1}{2}}\right) -\sum^{\mathcal{N}}_{n=2}\mu\tau_F\left(\Delta_{h} e_F^{n-\frac{1}{2}},e_F^{n-\frac{1}{2}}\right)- \sum^{\mathcal{N}}_{n=2}\mathfrak{w}_{n,1}\tau_F\left(\Delta_{h} e_F^{1},e_F^{n-\frac{1}{2}}\right) \\ &-\sum^{\mathcal{N}}_{n=2}\tau_F\sum_{p=2}^{n}\mathfrak{w}_{n,p}\left(\Delta_{h} e_F^{p-\frac{1}{2}},e_F^{n-\frac{1}{2}}\right)=
              \frac{\tau_F }{2}\sum^{\mathcal{N}}_{n=2}\left(g(u^{n-1})
            -g(U_F^{n-1}),e_F^{n-\frac{1}{2}}\right) \\&+\frac{\tau_F}{2}\sum^{\mathcal{N}}_{n=2} \left(g(u^n)-g((U_C)^n)-g'(U_C^n) \left(U_F^n-U_C^n\right),e_F^{n-\frac{1}{2}}\right)
        +\sum^{\mathcal{N}}_{n=2}\tau_F\left((R)^{n},e_F^{n-\frac{1}{2}}\right).
         \end{split}
       \end{equation}
       \vskip 0.2mm
       Adding (\ref{eq4.38}) and (\ref{eq4.39}), then using Lemmas \ref{lemma2.5}-\ref{lemma2.7}, Cauchy-Schwarz inequality and triangle inequality, and noting (\ref{eq4.36}), we can get
       \begin{equation}\label{eq4.40}
         \begin{split}
        \|e_F^{\mathcal{N}}\|^2&\leq\tau_F\|g(u^1)-g(U_C^1)-g'(U_C^1) \left(U_F^1-U_C^1\right)\|\|e_F^1\|+2\tau_F\|R^1\|\|e_F^1\| \\&+\tau_F\sum^{\mathcal{N}}_{n=2} \|g(u^n)-g(U_C^n)-g'(U_C^n) \left(U_F^n-U_C^n\right)\|\|e_F^{n-\frac{1}{2}}\|\\
        &+\bar{C}\tau_F\sum^{\mathcal{N}}_{n=2}\|e_F^{n-1}\|\|e_F^{n-\frac{1}{2}}\|+ 2\sum^{\mathcal{N}}_{n=2}\tau_F\|(R)^{n}\|\|e_F^{n-\frac{1}{2}}\|.
         \end{split}
       \end{equation}
       \vskip 0.2mm
       Choosing a suitable $\check{s}$ such that $\|e_F^{\check{s}}\|=\max\limits_{0\leq n\leq \mathcal{N}}\|e_F^n\|$, then it holds
       \begin{equation}\label{eq4.41}
         \begin{split}
            \|e_F^{\mathcal{N}}\|\leq \|e_F^{\check{s}}\|&\leq \tau_F\sum^{\mathcal{N}}_{n=1} \|g(u^n)-g(U_C^n)-g'(U_C^n) \left(U_F^n-U_C^n\right)\|\\
        &+\bar{C}\tau_F\sum^{\mathcal{N}}_{n=2}\|e_F^{n-1}\|+ 2\sum^{\mathcal{N}}_{n=1}\tau_F\|(R)^{n}\|.
         \end{split}
       \end{equation}
       \vskip 0.2mm
       According to Taylor expansion, we have
       \begin{equation}\label{eq4.42}
       \begin{split}
           &g(u^n)-g(U_C^n)-g'(U_C^n) \left(U_F^n-U_C^n\right)\\
           &= g'(U_C^{n})(u^{n} - U_C^{n} ) + \frac{1}{2}g''(\theta^n)(u^{n} - U_C^{n})^2  -  g'(U_C^{n})(U_F^{n} - U_C^{n} )  \\
       &= g'(U_C^{n})e_F^{n}  + \frac{1}{2}g''(\theta^n)(e_C^{n})^2,\quad \theta^n\in\big(\min \{u^n,U_C^n\},\max \{u^n,U_C^n\}\big).
       \end{split}
       \end{equation}
       \vskip 0.2mm
       Substituting (\ref{eq4.42}) into (\ref{eq4.41}) and applying the triangle inequality, we can get
       \begin{equation}\label{eq4.43}
         \begin{split}
            \|e_F^{\mathcal{N}}\|&\leq \bar{C}\tau_F\sum^{\mathcal{N}}_{n=1} (\|e_F^n\|+\|e_C^n\|^2)+\bar{C}\tau_F\sum^{\mathcal{N}}_{n=2}\|e_F^{n-1}\|+ 2\sum^{\mathcal{N}}_{n=1}\tau_F\|(R)^{n}\|\\
            &\leq \bar{C}\tau_F\sum^{\mathcal{N}}_{n=1} \|e_F^n\|+\bar{C}\tau_F\sum^{\mathcal{N}}_{n=1}\|e_C^n\|^2+ 2\sum^{\mathcal{N}}_{n=1}\tau_F\|(R)^{n}\|.
         \end{split}
       \end{equation}
       \vskip 0.2mm
       Utilizing Lemma \ref{lemma2.8} and Theorem \ref{th4.3}, we yield
       \begin{equation}\label{eq4.44}
         \begin{split}
             \|e_F^{\mathcal{N}}\|&\leq \bar{C}\exp\{\mathcal{N}\tau_F\} \left(\tau_F\sum^{\mathcal{N}}_{n=1}\|e_C^n\|^2+ \sum^{\mathcal{N}}_{n=1}\tau_F\|(R)^{n}\|\right) \\
              & \leq \bar{C}\left(\tau_C^4+\tau_F^2+h_1^2+h_2^2\right),
         \end{split}
       \end{equation}
       which completes the proof.
     \end{proof}
   \end{theorem}
  \section{Numerical experiment}\label{secc5}
     \vskip 0.2mm
     In this section, we will use the TTGCN finite difference scheme \eqref{eq3.9}-\eqref{eq3.13} to solve problem \eqref{eq1.1}-\eqref{eq1.3} and apply the method to three test problems. In order to verify the validity of the method, we also compare the results obtained from proposed scheme with the existing methods, e.g., the SCN finite difference scheme \eqref{eq3.5}-\eqref{eq3.8} and the scheme \cite{xu2020time}. We set $L_x=L_y=1$ and $T=1$. All experiments are performed on a Windows 11 (64 bit) PC-Inter(R) Core(TM) i5-12500H CPU 3.10 GHz, 16.0 GB of RAM using MTALAB R2021b.
     \vskip 0.2mm
     The discrete $L^2$-norm error is defined as follows
     \begin{equation*}
       E_{TTGCN}(h,\tau)=\max_{1\leq n\leq \mathcal{N}}\|u^n-U_F^n\|,
     \end{equation*}
     and the time-space convergence orders are defined by
     \begin{equation*}
       rate^t_{TTGCN}=\log_2\left(\frac{E_{TTGCN}(h,2\tau)}{E_{TTGCN}(h,\tau)}\right), \qquad rate^x_{TTGCN}=\log_2\left(\frac{E_{TTGCN}(2h,\tau)}{E_{TTGCN}(h,\tau)}\right).
     \end{equation*}
     In addition, we can similarly define $E_{SCN}(h,\tau)$, $rate^t_{SCN}$ and $rate^x_{SCN}$.

\begin{example}\label{ex1}
We consider the nonlinear term is given by $g(u)=-u^2$, $\mu=1$ and the inhomogeneous term is
     \begin{equation*}
         \begin{split}
            f(x,y,t)=& \left[(1+2\pi)\frac{t^{\alpha}}{\Gamma(\alpha+1)}+2\pi^{2}\left(1+\frac{ t^{\alpha+1}}{\Gamma(\alpha+2)}+\frac{t^{2\alpha+1}}{\Gamma(2\alpha+2)}\right)\right]\sin\pi x\sin\pi y\\
            &+\left(1+\frac{t^{\alpha+1}}{\Gamma(2+\alpha)}\sin\pi x\sin\pi y\right)^{2}.
         \end{split}
       \end{equation*}

       The exact solution of this problem is presented as follows
       \begin{equation*}\centering
         u(x,y,t)=\left(1+\frac{t^{\alpha+1}}{\Gamma(2+\alpha)}\right)\sin\pi x\sin\pi y.
       \end{equation*}
       \end{example}
       \vskip 0.2mm
       In Table \ref{tab:1}, we obtain the corresponding discrete $L^2$-norm errors, time convergence order and CPU time by calculating Example \ref{ex1} with the TTGCN finite difference scheme (\ref{eq3.9})-(\ref{eq3.13}) and the SCN finite difference method (\ref{eq3.5})-(\ref{eq3.8}). The numerical results show that the convergence order of the two schemes converges to 2 in the time direction, which is consistent with the theoretical analysis. Meanwhile, we compare the numerical results of the two methods in terms of temporal convergence order and computational cost (CPU time in seconds), and see that the TTGCN finite difference scheme can save computational cost significantly without losing computational accuracy.

 \begin{table}\small\centering
     	\caption{ The $L^2$-errors, convergence rates and CPU time (seconds) with $h=1/100$ and $k=4$ for Example \ref{ex1}.}
     \label{tab:1}  
     \begin{tabular}{ccccccccc}
      \hline\noalign{\smallskip}
     		$\alpha $    & $\tau_C$   & $\tau_F$  &  $E_{TTGCN}$    & $rate^{t}_{TTGCN}$    &$CPU(s)$  &  $E_{SCN}$    & $rate^{t}_{SCN}$   &$CPU(s)$  \\
     \noalign{\smallskip}\hline\noalign{\smallskip}
                         & 1/2     & 1/8    &  2.9293e-2    &  *                  & 41.42     &  2.9294e-2    &  *                  & 83.85    \\
			             & 1/4     & 1/16   &  9.9431e-3    &  1.5588             & 75.53     &  9.9431e-3    &  1.5588             & 159.79     \\
			  0.25       & 1/8     & 1/32   &  2.9743e-3    &  1.7412             & 176.76    &  2.9743e-3    &  1.7412             & 307.44      \\
		                 & 1/16    & 1/64   &  7.7382e-4    &  1.9425             & 439.63    &  7.7382e-4    &  1.9425             & 696.57      \\
     \noalign{\smallskip}\hline\noalign{\smallskip}
                         & 1/2     & 1/8    &  1.5390e-2    &  *                  & 35.26     &  1.5391e-2    &  *                  & 83.48    \\
			             & 1/4     & 1/16   &  4.2211e-3    &  1.8663             & 77.16     &  4.2212e-3    &  1.8664             & 160.76      \\
			 0.5         & 1/8     & 1/32   &  1.0102e-3    &  2.0630             & 177.43    &  1.0102e-3    &  2.0630             & 304.59    \\
		                 & 1/16    & 1/64   &  2.0588e-4    &  2.2948             & 441.57    &  2.0589e-4    &  2.2948             & 700.46      \\
     \noalign{\smallskip}\hline\noalign{\smallskip}
                         & 1/2     & 1/8    &  7.7023e-3    &  *                  & 35.64     &  7.7034e-3    &  *                  & 83.19    \\
			             & 1/4     & 1/16   &  1.7363e-3    &  2.1493             & 77.64     &  1.7364e-3    &  2.1494             & 159.14      \\
			 0.75        & 1/8     & 1/32   &  3.3266e-4    &  2.3839             & 176.50    &  3.3266e-4    &  2.3840             & 309.35    \\
		                 & 1/16    & 1/64   &  9.3963e-5    &  1.8239             & 414.11    &  9.3963e-5    &  1.8239             & 681.68      \\
     \noalign{\smallskip}\hline
     \end{tabular}
   \end{table}

   In addition, by the results in Table \ref{tab:2}, we can see that the TTGCN finite difference scheme will save more computational cost than the SCN finite difference scheme as the value of $k$ increases.

    \begin{table}\small\centering
     	\caption{The $L^2$-errors, convergence rates and CPU time (seconds) with $h=1/100$ and $\alpha=0.5$ for Example \ref{ex1}.}
     \label{tab:2}  
     \begin{tabular}{ccccccccc}
      \hline\noalign{\smallskip}
     		$k $    & $\tau_C$   & $\tau_F$  &  $E_{TTGCN}$    & $rate^{t}_{TTGCN}$    &$CPU(s)$  &  $E_{SCN}$    & $rate^{t}_{SCN}$   &$CPU(s)$  \\
     \noalign{\smallskip}\hline\noalign{\smallskip}
                         & 1/3     & 1/6    &  2.5490e-2    &  *                  & 40.93     &  2.5490e-2    &  *                  & 61.24    \\
			             & 1/6     & 1/12   &  7.3298e-3    &  1.7981             & 83.19     &  7.3299e-3    &  1.7981             & 120.88   \\
			 2           & 1/12    & 1/24   &  1.8622e-3    &  1.9767             & 181.48    &  1.8623e-3    &  1.9767             & 232.67   \\
		                 & 1/24    & 1/48   &  4.0706e-4    &  2.1937             & 391.92    &  4.0706e-4    &  2.1937             & 474.09   \\
     \noalign{\smallskip}\hline
                         & 1/2     & 1/6    &  2.5489e-2    &  *                  & 31.25     &  2.5490e-2    &  *                  & 61.91    \\
			             & 1/4     & 1/12   &  7.3297e-3    &  1.7980             & 65.04     &  7.3299e-3    &  1.7981             & 120.44   \\
			 3           & 1/8     & 1/24   &  1.8622e-3    &  1.9767             & 142.96    &  1.8623e-3    &  1.9767             & 232.99   \\
		                 & 1/16    & 1/48   &  4.0706e-4    &  2.1937             & 320.59    &  4.0706e-4    &  2.1937             & 479.59   \\
     \noalign{\smallskip}\hline
                         & 1/2     & 1/10   &  1.0281e-2    &  *                  & 40.04     &  1.0282e-2    &  *                  & 107.69    \\
			             & 1/4     & 1/20   &  2.7071e-3    &  1.9252             & 89.51     &  2.7071e-3    &  1.9253             & 198.24   \\
			 5           & 1/8     & 1/40   &  6.1692e-4    &  2.1336             & 214.16    &  6.1693e-4    &  2.1336             & 389.39   \\
		                 & 1/16    & 1/80   &  1.1915e-4    &  2.3723             & 531.10    &  1.1915e-4    &  2.3724             & 913.59   \\
     \noalign{\smallskip}\hline
     \end{tabular}
   \end{table}

    \vskip 0.2mm
       When the time step $\tau_C=1/128$ and $\tau_F=1/512$ are fixed, in Tables \ref{tab:3}, the convergence order of the two schemes in space is 2 according to the numerical results. Therefore, the convergence results in the space-time direction are in good agreement with the theoretical analysis.
 \begin{table}\small\centering
     	\caption{The $L^2$-errors and convergence rates with $\tau_C=1/128$ and $\tau_F=1/512$ for Example \ref{ex1}.}
       \label{tab:3}       
     	\begin{tabular}{ccccccccc}
        \hline\noalign{\smallskip}
     		$\alpha $  & $h$ &  $E_{TTGCN}$ & $rate^{x}_{TTGCN}$        &  $E_{SCN}$ & $rate^{x}_{SCN}$   \\
       \noalign{\smallskip}\hline\noalign{\smallskip}
     			             & 1/2   &  3.8785e-1    &  *            &  3.8785e-1   &  -             \\
			                 & 1/4   &  9.1992e-2    &  2.0759       &  9.1992e-2   &  2.0759          \\
			  0.20           & 1/8   &  2.2621e-2    &  2.0239       &  2.2621e-2   &  2.0239           \\
			                 & 1/16  &  5.6309e-3    &  2.0062       &  5.6309e-3   &  2.0062          \\
			                 & 1/32  &  1.4061e-3    &  2.0017       &  1.4061e-3   &  2.0017          \\
       \noalign{\smallskip}\hline\noalign{\smallskip}
			                 & 1/2   &  3.5774e-1    &   *            &  3.5774e-1    &   *            \\
			                 & 1/4   &  8.4731e-2    &  2.0780        &  8.4731e-2    &  2.0780         \\
			  0.50           & 1/8   &  2.0830e-2    &  2.0242        &  2.0830e-2    &  2.0242        \\
			                 & 1/16  &  5.1848e-3    &  2.0063        &  5.1848e-3    &  2.0063        \\
			                 & 1/32  &  1.2945e-3    &  2.0018        &  1.2945e-3    &  2.0018         \\
       \noalign{\smallskip}\hline\noalign{\smallskip}
			                 & 1/2   &  3.2854e-1    &   *           &  3.2854e-1    &   *           \\
			                 & 1/4   &  7.7643e-2    &  2.0811       &  7.7643e-2    &  2.0811        \\
			  0.80           & 1/8   &  1.9081e-2    &  2.0248       &  1.9081e-2    &  2.0248        \\
			                 & 1/16  &  4.7487e-3    &  2.0065       &  4.7487e-3    &  2.0065        \\
			                 & 1/32  &  1.1855e-3    &  2.0020       &  1.1855e-3    &  2.0020         \\
       \noalign{\smallskip}\hline
     	\end{tabular}
     \end{table}

     Fig. 1 compares the computation time of the two-grid method and the standard method in the time direction for the Crank-Nicolson finite difference scheme. It can be observed that the computational cost of the TTGCN finite difference method is lower without losing the accuracy. Also, Fig. 2 gives the $L^2$-norm error for both methods, which can show intuitively second-order convergence for time.

    \begin{figure}
    \renewcommand{\figurename}{Fig.}\centering
    \includegraphics[width=0.7\textwidth]{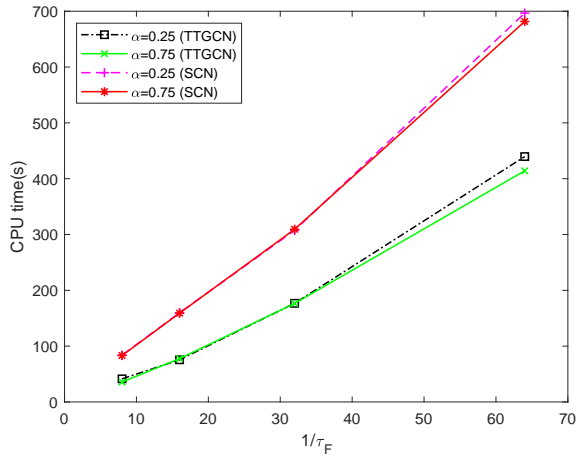}
    \caption {The comparison of two methods for CPU time with $h=1/100$ and $k=4$ for  Example \ref{ex1}.}
    \end{figure}
    \begin{figure}
    \renewcommand{\figurename}{Fig.}\centering
    \includegraphics[width=0.7\textwidth]{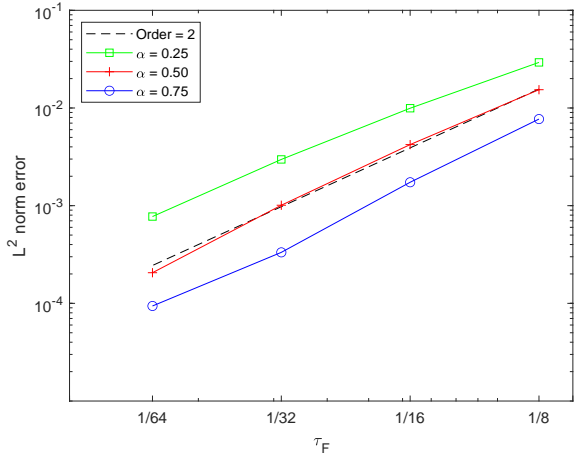}
    \caption {The time convergence order with $h=1/100$ and $k=4$ for Example \ref{ex1}.}
    \end{figure}

    \vskip 0.2mm

  \begin{example}\label{ex2}
   we consider $g(u)=-u-u^3$ and $\mu=1$. The exact solution is given via
   \begin{equation*}
     u(x,y,t)=\frac{t^{\alpha+1}}{\Gamma(2+\alpha)}\sin\pi x\sin\pi y,
   \end{equation*}
   thus, $\psi(x,y)=0$ and the corresponding force term can be obtained as follows
   \begin{equation*}
   \begin{split}
     f(x,y,t)=&\left(\frac{t^{\alpha}}{\Gamma(\alpha+1)}+\frac{(2\pi^2\mu+1) t^{\alpha+1}}{\Gamma(\alpha+2)}+\frac{t^{2\alpha +1}}{\Gamma(2\alpha+2)}\right)\sin\pi x\sin\pi y \\ &+\left(\frac{t^{\alpha+1}}{\Gamma(2+\alpha)}\sin\pi x\sin\pi y\right)^3.
   \end{split}
   \end{equation*}
  \end{example}
   In Table \ref{tab:4}, we give the numerical results with $\alpha=0.25$, $0.5$ and $0.75$ calculated using the TTGCN finite difference method and the SCN finite difference method, respectively. This numerical result fully demonstrates that the computational efficiency of the TTGCN finite difference method is much higher than that of the SCN finite difference method. Also, according to the numerical results in Table \ref{tab:5}, the order of convergence of the two methods in space $\thickapprox 2$. Therefore, the numerical results are consistent with the theoretical analysis. In addition, we also compared with the method in \cite{xu2020time}. It is obvious from Table \ref{tab:6} that the TTGCN finite difference method has higher accuracy and convergence order.

  When $h=1/100$ and $k=4$, Fig. 3 compares the CPU time of the two-grid finite difference method and the standard finite difference method for the time direction, which intuitively demonstrates the effectiveness of our method. Besides, Fig. 4 shows intuitively temporal second-order convergence of two-grid finite difference method.

   \begin{table}\small\centering
     	\caption{The $L^2$-errors, convergence rates and CPU time (seconds) with $h=1/100$ and $k=4$ for Example \ref{ex2}.}
     \label{tab:4}  
     \begin{tabular}{ccccccccc}
      \hline\noalign{\smallskip}
     		$\alpha $    & $\tau_C$   & $\tau_F$  &  $E_{TTGCN}$    & $rate^{t}_{TTGCN}$    &$CPU(s)$  &  $E_{SCN}$    & $rate^{t}_{SCN}$   &$CPU(s)$  \\
     \noalign{\smallskip}\hline\noalign{\smallskip}
                         & 1/2     & 1/8    &  2.9535e-2    &  *                  & 34.31     &  2.9535e-2    &  *                  & 64.26    \\
			             & 1/4     & 1/16   &  1.0043e-2    &  1.5563             & 71.10     &  1.0043e-2    &  1.5563             & 125.59     \\
			  0.25       & 1/8     & 1/32   &  3.0208e-3    &  1.7331             & 161.96    &  3.0208e-3    &  1.7331             & 264.99      \\
		                 & 1/16    & 1/64   &  7.9828e-4    &  1.9200             & 412.81    &  7.9828e-4    &  1.9200             & 561.79      \\
     \noalign{\smallskip}\hline\noalign{\smallskip}
                         & 1/2     & 1/8    &  1.5532e-2    &  *                  & 32.46     &  1.5532e-2    &  *                  & 58.22    \\
			             & 1/4     & 1/16   &  4.2840e-3    &  1.8582             & 69.57     &  4.2840e-3    &  1.8582             & 120.92      \\
			 0.5         & 1/8     & 1/32   &  1.0448e-3    &  2.0357             & 156.46    &  1.0448e-3    &  2.0357             & 242.71    \\
		                 & 1/16    & 1/64   &  2.2614e-4    &  2.2080             & 391.87    &  2.2614e-4    &  2.2080             & 561.08      \\
     \noalign{\smallskip}\hline\noalign{\smallskip}
                         & 1/2     & 1/8    &  7.7945e-3    &  *                  & 30.86     &  7.7946e-3    &  *                  & 59.13    \\
			             & 1/4     & 1/16   &  1.7861e-3    &  2.1256             & 67.44     &  1.7861e-3    &  2.1256             & 113.23      \\
			 0.75        & 1/8     & 1/32   &  3.6423e-4    &  2.2939             & 156.47    &  3.6423e-4    &  2.2939             & 233.35    \\
		                 & 1/16    & 1/64   &  6.6431e-5    &  2.4549             & 402.73    &  6.6431e-5    &  2.4549             & 545.91      \\
     \noalign{\smallskip}\hline
     \end{tabular}
   \end{table}

        \begin{table}\small
     	\caption{The $L^2$-errors and convergence rates with $\tau_C=1/128$ and $\tau_F=1/512$ for Example \ref{ex2}.}
       \label{tab:5}       
     	\resizebox{\textwidth}{!}{
     \begin{tabular}{ccccccccccc}
        \hline\noalign{\smallskip}
             \multirow{2}{*}{$h$} &\multicolumn{4}{c}{$\alpha=0.2$} &&  \multicolumn{4}{c}{$\alpha=0.8$}      \\
              \cline{2-5} \cline{7-10}
              & $E_{TTGCN}$ & $rate^{x}_{TTGCN}$    &  $E_{SCN}$ & $rate^{x}_{SCN}$ && $E_{TTGCN}$ & $rate^{x}_{TTGCN}$  &  $E_{SCN}$ & $rate^{x}_{SCN}$ \\
       \noalign{\smallskip}\hline\noalign{\smallskip}
     	1/2   &  1.8132e-1    &  *            &  1.8132e-1    &  *      &&  1.1940e-1    &   *           &  1.1941e-1    &   *         \\
		1/4   &  4.3294e-2    &  2.0663       &  4.3296e-2    &  2.0663 &&  2.8125e-2    &  2.0859       &  2.8133e-2    &  2.0856     \\
		1/8   &  1.0650e-2    &  2.0233       &  1.0652e-2    &  2.0231 &&  6.9065e-3    &  2.0259       &  6.9139e-2    &  2.0247     \\
		1/16  &  2.6486e-3    &  2.0070       &  2.6519e-3    &  2.0060 &&  1.7133e-3    &  2.0112       &  1.7206e-3    &  2.0066     \\
		1/32  &  6.5992e-4    &  2.0054       &  6.6217e-4    &  2.0017 &&  4.2551e-4    &  2.0095       &  4.2932e-4    &  2.0028     \\
       \noalign{\smallskip}\hline
     	\end{tabular}}
     \end{table}

    \begin{figure}
    \renewcommand{\figurename}{Fig.}\centering
    \includegraphics[width=0.7\textwidth]{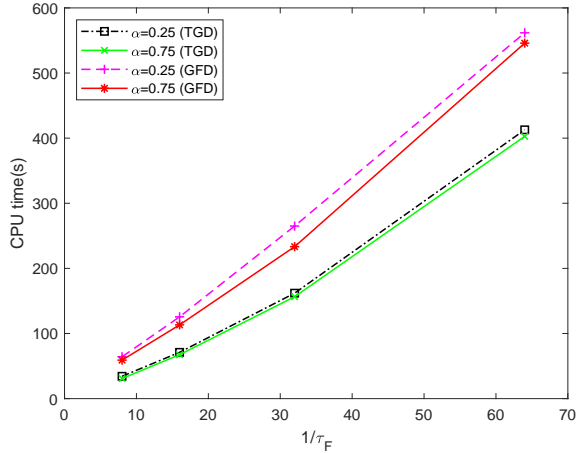}
    \caption {The CPU time for Example \ref{ex2} with $h=1/100$ and $k=4$.}
    \end{figure}
    \begin{figure}
    \renewcommand{\figurename}{Fig.}\centering
    \includegraphics[width=0.7\textwidth]{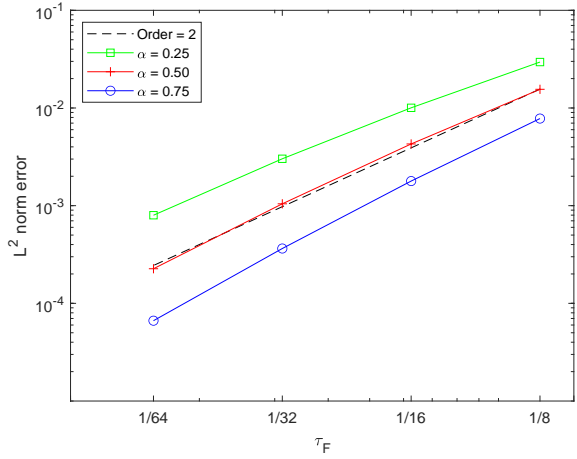}
    \caption {The time convergence order for Example \ref{ex2} with $h=1/100$ and $k=4$.}
    \end{figure}

    \begin{table}\small\centering
     	\caption{The comparison between the scheme (\ref{eq3.9})-(\ref{eq3.13}) and the scheme \cite{xu2020time} whit $h=1/100$ and $k=4$ for Example \ref{ex2}.}
     \label{tab:6}  
     \begin{tabular}{cccccccccc}
      \hline\noalign{\smallskip}
     	       \multirow{2}{*}{$\alpha$} & \multirow{2}{*}{$\tau_C$}   & \multirow{2}{*}{$\tau_F$}  & &    \multicolumn{2}{c}{Scheme (\ref{eq3.9})-(\ref{eq3.13})} & &
      \multicolumn{2}{c}{Scheme in \cite{xu2020time}}  \\
     \cline{5-6} \cline{8-9}
                         &         &        & &     $E_{TTGCN}$   & $rate^{t}_{TTGCN}$        &   &$E$        &$rate^{t}$   \\
     \noalign{\smallskip}\hline\noalign{\smallskip}
                         & 1/2     & 1/8    & & 2.9535e-2    &  *         &    &  3.9266e-3    &  *                      \\
			             & 1/4     & 1/16   & &1.0043e-2     &  1.5563    &    &  1.9639e-3    &  0.9996                  \\
			  0.25       & 1/8     & 1/32   & & 3.0208e-3    &  1.7331    &    &  9.7001e-4    &  1.0176                   \\
		                 & 1/16    & 1/64   & & 7.9828e-4    &  1.9200    &    &  4.6979e-4    &  1.0460                   \\
     \noalign{\smallskip}\hline\noalign{\smallskip}
                         & 1/2     & 1/8    & & 1.5532e-2    &  *         &    &  7.6809e-3    &  *                      \\
			             & 1/4     & 1/16   & & 4.2840e-3    &  1.8582    &    &  3.8683e-3    &  0.9896                   \\
			 0.5         & 1/8     & 1/32   & & 1.0448e-3    &  2.0357    &    &  1.9311e-3    &  1.0023                 \\
		                 & 1/16    & 1/64   & & 2.2614e-4    &  2.2080    &    &  9.5444e-4    &  1.0167                   \\
     \noalign{\smallskip}\hline\noalign{\smallskip}
                         & 1/2     & 1/8    & & 7.7945e-3    &  *         &    &  9.7620e-3    &  *                     \\
			             & 1/4     & 1/16   & & 1.7861e-3    &  2.1256    &    &  4.9287e-3    &  0.9860                   \\
			 0.75        & 1/8     & 1/32   & & 3.6423e-4    &  2.2939    &    &  2.2683e-3    &  0.9977               \\
		                 & 1/16    & 1/64   & & 6.6431e-5    &  2.4549    &    &  1.2266e-3    &  1.0088                  \\
     \noalign{\smallskip}\hline
     \end{tabular}
   \end{table}

    \vskip 0.2mm
    \begin{example}\label{ex3}
      we consider
    \begin{equation*}
      \begin{aligned}
        &u_{t}-\Delta u-I^{(\alpha)}\Delta u=-u^3,\quad (x,y,t)\in\Omega \times (0,T], \\
       & u(x,y,t)=0,\qquad (x,y)\in\partial\Omega,\quad t\in [0,T], \\
        &u(x,y,0)=xy(1-x)(1-y),\qquad (x,y)\in\Omega.
  \end{aligned}
    \end{equation*}
    \end{example}
    In this example, since the exact solution is unknown, we assume that the numerical solution with fixed spatial step $h=1/32$ and half of the original time steps $\tau_C$ and $\tau_F$ is the ``exact" solution. From Table \ref{tab:7}, we can see that for the time direction convergence order TTGCN and SCN finite difference methods in both can approach 2, which agrees with the theoretical analysis.

    \begin{table}\small\centering
     	\caption{The $L^2$-errors and convergence rates with $h=1/32$ and $k=4$ for Example \ref{ex3}}
     \label{tab:7}  
     \begin{tabular}{cccccccc}
      \hline\noalign{\smallskip}
     		$\alpha $    & $\tau_C$   & $\tau_F$  &  $E_{TTGCN}$    & $rate^{t}_{TTGCN}$      &  $E_{SCN}$    & $rate^{t}_{SCN}$     \\
     \noalign{\smallskip}\hline\noalign{\smallskip}
                         & 1/12     & 1/48    &  6.0750e-7    &  *                &  6.0563e-7       &  *                  \\
			             & 1/24     & 1/96    &  1.8258e-7    &  1.7344           &  1.8214e-7       &  1.7334             \\
			  0.25       & 1/48     & 1/192   &  4.9624e-8    &  1.8794           &  4.9558e-8       &  1.8779             \\
		                 & 1/96     & 1/384   &  1.2745e-8    &  1.9611           &  1.2739e-8       &  1.9599             \\
     \noalign{\smallskip}\hline\noalign{\smallskip}
                         & 1/12     & 1/48    &  1.2281e-6    &  *                &  1.2246e-6       &  *                  \\
			             & 1/24     & 1/96    &  3.6180e-7    &  1.7661           &  3.6036e-7       &  1.7648             \\
			  0.5        & 1/48     & 1/192   &  9.7691e-8    &  1.8860           &  9.7595e-8       &  1.8846             \\
		                 & 1/96     & 1/384   &  2.5272e-8    &  1.9507           &  2.5363e-8       &  1.9498             \\
     \noalign{\smallskip}\hline\noalign{\smallskip}
                         & 1/12     & 1/48    &  2.1532e-6    &  *                &  2.1477e-6       &  *                  \\
			             & 1/24     & 1/96    &  6.3459e-7    &  1.7626           &  6.3351e-7       &  1.7614             \\
			  0.75       & 1/48     & 1/192   &  1.7308e-7    &  1.8744           &  1.7294e-7       &  1.8731             \\
		                 & 1/96     & 1/384   &  4.5190e-8    &  1.9373           &  4.5177e-8       &  1.9366             \\
     \noalign{\smallskip}\hline
     \end{tabular}
   \end{table}

\section*{Declaration of Competing Interest}
The authors declare that they have no conflict of interest.
\section*{Acknowledgment}
The project was supported by Postgraduate Scientific Research Innovation Project of Hunan Province (No. CX20220469).

\bibliographystyle{unsrt}
\bibliography{Bibfile}

\end{document}